\documentclass[12pt,reqno]{amsart}
\usepackage[utf8]{inputenc}
\usepackage[T1]{fontenc}
\usepackage{csquotes}
\usepackage{amsmath,amssymb,amsthm,latexsym}
\usepackage{mathtools}
\usepackage{geometry}
\usepackage[expansion=false]{microtype}
\usepackage{enumitem}
\usepackage{tabularx}
\usepackage{booktabs}
\usepackage{tikz-cd}
\usepackage{xcolor}
\usepackage[square,numbers]{natbib}
\usepackage{hyperref}
\hypersetup{colorlinks=true,linkcolor=blue!50!black,citecolor=green!40!black,
            urlcolor=blue!60!black}
\usepackage[nameinlink]{cleveref}
\crefalias{align}{equation}
\crefalias{gather}{equation}
\crefalias{multline}{equation}
\crefalias{alignat}{equation}

\crefformat{equation}{(#2#1#3)}

\crefname{theorem}{theorem}{theorems}
\Crefname{theorem}{Theorem}{Theorems}

\crefname{lemma}{lemma}{lemmas}
\Crefname{lemma}{Lemma}{Lemmas}

\crefname{proposition}{proposition}{propositions}
\Crefname{proposition}{Proposition}{Propositions}

\crefname{corollary}{corollary}{corollaries}
\Crefname{corollary}{Corollary}{Corollaries}

\crefname{definition}{definition}{definitions}
\Crefname{definition}{Definition}{Definitions}
\allowdisplaybreaks

\newcommand{\A}{\mathbb{A}}                 
\newcommand{\Mon}{\mathcal{M}}              
\newcommand{\ind}{\mathbf{1}}               
\newcommand{\Pam}{\mathcal{P}_{a,M}}
\newcommand{\Eul}{\Phi}                     
\newcommand{\dTV}{d_{\mathrm{TV}}}
\DeclareMathOperator{\Poisson}{Poisson}
\DeclareMathOperator{\Bor}{Bor}
\newcommand{\Ecal}{\mathcal{E}}             
\newcommand{\Fset}{E^{*}_{S}}

\theoremstyle{plain}
\newtheorem{theorem}{Theorem}[section]

\newtheorem{lemma}[theorem]{Lemma}
\newtheorem{proposition}[theorem]{Proposition}
\newtheorem*{conjecture}{Conjecture}

\theoremstyle{definition}
\newtheorem{definition}[theorem]{Definition}

\theoremstyle{remark}
\newtheorem*{remark}{Remark}

\begin{document}

\title{\bf Chowla Conjecture for the truncated M\"obius function over $\mathbb{F}_q[t]$}
\author{Abhirup Chatterjee$^{1,2}$}
	\address{$^{1}$The Institute of Mathematical Sciences, CIT Campus, Taramani, Chennai 600113, India}  

    \address{$^{2}$Homi Bhabha National Institute, Training School Complex, Anushakti Nagar, Mumbai 400094, India}

	\email[Abhirup Chatterjee]{abhirupc@imsc.res.in}
	

\begin{abstract}
In this paper we give an analytic proof of the Two-point Chowla conjecture for the
truncated M\"obius function $\mu_M$ for the function field case in the fixed $q$ setting. We reformulate the problem into studying the distribution of $(\omega_M,\omega_M\circ T_a)$. Thus we develop a uniform bivariate
Erd\H{o}s-Kac theorem in terms of the characteristic function for small pair of real numbers $(u,v)$ via the Kubilius model.
For this purpose we also develop and employ the pure Brun sieve
for polynomials in this paper.
\end{abstract}
	\maketitle

\section{Introduction}
\label{sec:intro}

In problem~57 of the book \cite[p.~341]{ChowlaConj1965} Sarvadaman Chowla conjectured
the following.

\begin{conjecture}[Chowla's Conjecture for the Liouville function]
For any integer polynomial $F(x)$ that is not of the form $c\,g(x)^2$ for some
integer polynomial $g(x)$, one has
\[
\sum_{n\leq x} \lambda(F(n)) = o(x),
\]
where $\lambda$ is the Liouville function.
\end{conjecture}

When $F$ is linear, this is equivalent to the prime number
theorem on arithmetic progression. For polynomials of higher degree this is still open. The conjecture
has many different formulations and variations mostly in one of the two forms $F(*)=\prod_{i=1}^k(*\ +h_i)$(shift form) or $F(*)=\prod_{i=1}^k(a_i* +b_i)$(affine form)(see \cite{Ramare2018}). In the classical setup the most significant work has been in \cite{Matom_ki_2015},\cite{TAO_2016},\cite{Tao_2019} where they respectively prove the result for the shift form on average over the shifts \cite{Matom_ki_2015}, prove a 2-point logarithmic version in the affine form \cite{TAO_2016} and prove a k-point logarithmic version for odd k in the affine form \cite{Tao_2019} ; all of these results essentially rely on the breakthrough work of Matom\"aki and
Radziwi\l{}\l{} \cite{matomaki2017multiplicativefunctionsshortintervals}. Furthermore, the work of Mangerel in \cite{MANGEREL_2020} resolves the problem
for the truncated M\"obius function $\mu_y$, which, even though
\emph{`independent'} of the classical M\"obius function, is closely related to
it. In the function field setup Carmon and Rudnick solves the conjecture in shift form for $deg(f)$ fixed and $q\rightarrow\infty$ case (\cite{Carmon_Rudnick},\cite{Carmon_2015}). And Sawin and Shusterman resolves the conjecture in shift form for $q$ fixed (but large) and $deg(f)\rightarrow\infty$ case (\cite{sawin2019chowlatwinprimesconjectures}).\\
 The work of Mangerel \cite{MANGEREL_2020} in the classical setup is of great importance to us, as we employ the same
techniques used there to prove its
 counterpart for
the function field case. In particular, if $\Mon \subset \mathbb{F}_q[t] $ be the set of all monics, $f\in \mathbb{F}_q[t]$ and
$M\in \mathbb{N}$ then we define
\[
\omega(f):=\#\{p\in\Mon \mid p \text{ is irreducible and } p \text{ divides } f \},
\]
\[
\omega_M(f):=\#\{p\in\Mon \mid p \text{ is irreducible, } p \text{ divides } f
\text{ and } \deg p \le M \},
\]
\[
\mu(f):=
\begin{cases}
    (-1)^{\omega(f)} & \text{if $f$ is squarefree,} \\
    0  & \text{otherwise,}
\end{cases}
\]
\[
\mu_M(f):=\mu^2(f)(-1)^{\omega_M(f)}.
\]
It is not difficult to see that $\mu_M$ is essentially
independent of $\mu$ due to the pairity barrier. We
explain this in a more quantitative form in the appendix.

With these definitions in hand our main theorem is as follows.

\begin{theorem}
\label{main_thm}
Let $q\ne 2$ . Fix $a\in \Mon$. There is a constant $C$ such that, $\forall$ large $N$ and $M$ such that
\begin{equation}
\label{main_thm_range}
M\ \le\ \frac{N}{C\log N},
\end{equation}
we have
    \begin{equation}
    \label{main_thm_eq_1}
     \frac{1}{q^N}\sum_{\substack{\deg f\le N \\ f\in \Mon}} \mu_{M}(f)\mu_{M}(f+a)
     \ \ll_a\ \frac{1}{\log M}.
    \end{equation}
\end{theorem}

We achieve this result by noticing the crucial observation that the absolute
value of the left-hand side of \Cref{main_thm_eq_1} is in essence the
characteristic function of the random vector $(\omega_M,\omega_M\circ T_a)$, where
$T_a(f):=f+a$. Thus our problem boils down to understanding the
distribution of $(\omega_M,\omega_M\circ T_a)$. This naturally leads us to
seek an Erd\H{o}s--Kac type theorem. More precisely, we prove the following in
this paper.

\begin{theorem}[restated and proved as \cref{Erdos_Kac} below]
\label{Erdos_Kac_intro}

    Let $u,v\in\mathbb{R}$ with $|u|,|v|\le \pi \sqrt{\lambda_\rho(M)}$ . Then for any $\Delta>0$ $\exists$ $C=C(\Delta,\theta,q)$ such that for $M\le N/(C\log N)$
    \[
    |\varphi^*_{N,S,a}(u,v;M)-\chi(u,v)|\ll e^{-\frac{1}{2}(u^2+v^2)}
    \min\left\{\!1,\frac{|u|^3+|v|^3}{\sqrt{\lambda_\rho(M)}}\right\}
    + \min\left\{\!1,\frac{1}{\lambda_\rho(M)}\right\} +M^{-\Delta}+\frac{1}{N}
    \]
    where $\chi(u,v)=e^{-\frac{1}{2}(u^2+v^2)}$ is the characteristic function
    of the bivariate standard normal distribution ,
    $\varphi^*_{N,S,a}(u,v;M)$ is the characteristic function of the
    normalised random vector $(\widetilde{\omega}_{M},
    \widetilde{\omega}_{M}\circ T_a)$ over `doubly' square-free monics of degree $\le N$ (to be defined precisely later).
\end{theorem}
\begin{remark}
    The $M^{-\Delta}$ and the upper bound on $M$ comes directly from the Pure Brun Sieve \cref{sieve_main_thm}. A better sieve will result in better error and range of $M$.
\end{remark}
This version of Erd\H{o}s--Kac is in terms of the characteristic function (as
that is what we need). One should note that even though a distribution-function
type Erd\H{o}s--Kac theorem (which should be more readily available and
\textbf{does not} require a truncation), in addition to L\'evy's continuity
theorem, gives us an Erd\H{o}s--Kac type result in terms of the characteristic function, such
an estimate will only be pointwise in nature. However, uniform continuity, at
least in the small region mentioned above, is
indispensable in
proving \cref{main_thm}. Thus we take the longer approach and employ the
\emph{Kubilius model} in proving \cref{Erdos_Kac}, similarly to what is done in
\cite{MANGEREL_2020}. Also, for the sake of completeness we have added Brun
sieve theory for polynomials in the appendix.

Now the process of proving the \emph{Bivariate Erd\H{o}s--Kac} theorem for function fields can be
tracked by the following schematic map.

\begin{figure}[h]
\centering
\begin{tikzcd}[column sep=6.0em, row sep=4.5em]
(\widetilde{\omega}_M, \widetilde{\omega}_M\circ T_a)
\arrow[r,"\text{Kubilius model}","\text{(close in measure)}"']
\arrow[d,"{\substack{\text{Corresponding}\\ \text{characteristic}\\ \text{function}}}"']
&
\Sigma_M
\arrow[r,"\text{Poisson approx.}","\text{(close in measure)}"']
&
Z
\arrow[d,"{\substack{\text{Corresponding}\\ \text{characteristic}\\ \text{function}}}"]
\\
\varphi_{(\widetilde{\omega}_M, \widetilde{\omega}_M\circ T_a)}
\arrow[rr,"\approx"']
&&
\varphi_Z
\arrow[r,"\text{Poisson}\approx\text{Normal}"]
&
\varphi_{\mathcal N(0,1)}
\end{tikzcd}
\caption{Schematic of the proof of the bivariate Erd\H{o}s--Kac theorem.}
\label{fig:schematic}
\end{figure}
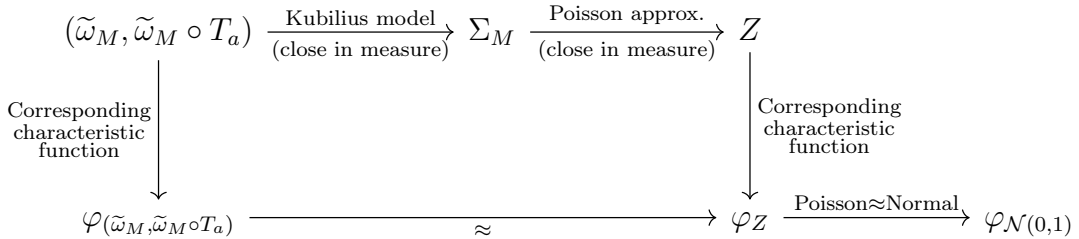

\noindent Here $\Sigma_M$ is a sum of bivariate Bernoulli vectors required for
the Kubilius model, and $Z$ is its Poisson approximation.

\section{Notation and conventions}
\label{sec:Notation}

We take $q\ne 2$. In this paper, whenever we say ``\emph{a prime}'' we mean \emph{a monic
irreducible polynomial}, unless stated otherwise. Furthermore if $d\mid f$ and $f\in \Mon$, upto a multiplication by a unit, we consider $d\in \Mon$. We do the same for g.c.d. or l.c.m. 

\begin{center}
\begin{tabularx}{0.92\textwidth}{@{}l>{\raggedright\arraybackslash}X@{}}
\toprule
 Notation & Meaning \\
\midrule
 $\mathbb{F}_q$   &  A field with $q$ elements \\
 $\A:=\mathbb{F}_q[t]$   &  Polynomial ring over $\mathbb{F}_q$ \\
 $\Mon$   &  Set of monic polynomials \\
 $S$ & $\{f\in \Mon \mid \mu^2(f)=1\}$,i.e, set of square-free monics\\
 $|f|$ & $q^{\deg f}$ \\
 $\Eul(f)$ & Euler totient of $f$ \\
 $\tau(f)$ & Number of monic divisors of $f$ \\
  \enquote{ $\asymp$ }   & $\psi_1\asymp\psi_2 \implies$ $\psi_1 \ll \psi_2$ and $\psi_2 \ll \psi_1$\\
 $\ind_{X}$ & Indicator of the condition $X$ \\
 $\mathcal{P}_a$ \ \ (where $a\in \Mon$)  &
   $\{ p\in \Mon \mid p \text{ is a prime}, \ p\nmid a \}$ \\
 $\Pam$ \ \ (where $a\in \Mon$)  &
   $\prod\limits_{p\in\mathcal{P}_a,\ \deg p\le M}p$ \ (a \emph{polynomial}, not a set) \\
 $\theta$ & The level of distribution (\Cref{def:siftable}) \\
 $\lambda_\rho(M)$ & $\sum_{\deg p\le M}\rho(p)$ \\
 $T_a$ & The shift $T_a(f)=f+a$ \\
 $\gamma^{(r)},\gamma^{\pm}$ & Brun sieve weights (\Cref{sec:brun}) \\
 $\nu_M,\ \sigma_M$ & The counting measure and its Kubilius model
   (\Cref{sec:Mass}) \\
 $\varphi^*_{N,S,a}$ & Characteristic function of
   $(\widetilde\omega_M,\widetilde\omega_M\circ T_a)$ \\
 $\chi$ & Characteristic function of the bivariate standard normal \\
\bottomrule
\end{tabularx}
\end{center}

Throughout, $\log$ denotes the natural logarithm and $\log\log$ the iterated
logarithm.

We shall employ the usual conventions of probability theory. We denote by
$\mathbb{P}$ a fixed probability measure on a measurable space
$(\Omega,\mathcal{B})$ on which a given random vector
$X:\Omega\to\mathbb{R}^n$ is defined. If $V\subseteq\mathbb{R}^n$ is Borel
measurable and $X^{-1}(V)=:A\in\mathcal{B}$, then we will write
$\mathbb{P}(\{X\in V\})$ in place of $\mathbb{P}(A)$.

We let $\mathbb{E}X$ denote the expectation of $X$, i.e.
\[
\mathbb{E}X := \int_{\Omega} X\,d\mathbb{P}
= \int_{\mathbb{R}^n} t\,d\mathcal{L}(X)(t),
\]
where we have written $\mathcal{L}(X):=\mathbb{P}\circ X^{-1}$ to denote the
law of~$X$.

\section{Bivariate Siftability of squarefree polynomials}
\label{sec:Struc}

As mentioned earlier, we would like to develop a Kubilius model for our random
vector $(\widetilde{\omega}_{M},\widetilde{\omega}_{M}\circ T_a)$. For this we
require a bit of sieving to understand the measure of this random vector. In
particular we require the basic sieving
criteria to hold in the
bivariate polynomial case.

\begin{definition}
\label{def:siftable}
    Let $a\in \Mon$ and $\mathcal{A} \subset \Mon$. We say that $\mathcal{A}$ is \emph{siftable
    with respect to $a$} if:
    \begin{enumerate}[label=(\roman*)]
        \item We have
        \[
        \sum_{\substack{\deg f \le N \\ (f,a)=1}} \ind_\mathcal{A}(f)\ind_\mathcal{A}(f+a) \gg_a q^N;
        \]
        \item There exists a non-negative multiplicative function
        $\rho=\rho_{\mathcal{A},a}$ and a real number $\theta\in(0,1]$, called the
        \emph{level of distribution}, such that
        \begin{enumerate}[label=(\alph*)]
            \item $0\le \rho(p)\le \dfrac{1}{|p|}$ for each prime $p$;
            \item for each pair of coprime
            squarefree
            polynomials $k_1,k_2$ with $|k_1|,|k_2| <q^N$ we have
            \[
            \sum_{\substack{\deg f \le N \\ f\equiv 0 \pmod{k_1} \\
            f+a\equiv 0 \pmod{k_2}}} \ind_\mathcal{A}(f)\ind_\mathcal{A}(f+a)\ind_{(f,a)=1}
            \]
            \[
            \qquad = \rho (k_1k_2) \left(\sum_{\substack{\deg f \le N \\ (f,a)=1}}
            \ind_\mathcal{A}(f)\ind_\mathcal{A}(f+a)+ R_a(N;k_1,k_2) \right)
            \]
            such that the remainders $R_a(N;k_1,k_2)$ satisfy
            \[
            \sum_{\substack{|k_1k_2| \le q^{N\theta} \\ \mu^2(k_1k_2)=1}}
            \rho (k_1k_2) |R_a(N;k_1,k_2)|\ind_{(k_1k_2,a)=1}
            \ll \frac{q^N}{N^3}.
            \]
            
        \end{enumerate}
    \end{enumerate}
    \textbf{Moreover}, we say that $\mathcal{A}$ is \emph{siftable} if it is siftable
    with respect to each $a\in \Mon$ with the same density function and level
    of distribution.
\end{definition}

\begin{definition}
    We say that a siftable set $\mathcal{A}$ is \emph{regular} if
    $\lambda_\rho(M) := \sum_{\deg p\le M}\rho(p) \asymp \log M$ for all
    $M \in \mathbb{N}^{*}$.
\end{definition}

Our goal for this section is to prove the following.

\begin{theorem}
    \label{siftability}
    The set $S:=\{f\in \Mon \mid \mu^2(f)=1\}$ of squarefree polynomials is
    siftable and regular with the corresponding density function 
    \begin{align*}
\rho(m)
= \frac{1}{|m|}\prod_{p\mid m}\left(1-\frac{|p|-2}{|p|^2-2}\right),
\end{align*} 
and any fixed $\theta \in (0,\frac{1}{3})$ as the level of distribution.
\end{theorem}

Write
\[
\mathcal{Q}_N:=\#\{f\in\Mon:\deg f\le N\}=\frac{q^{N+1}-1}{q-1},
\]
and for each $h\in \Mon$ let
\begin{align*}
    E_h(N)&:= \mathcal{Q}_N\prod_{p \nmid h}(1-2|p|^{-2}), \\
    E(N)&:=E_1(N).
    \end{align*}
    \begin{lemma}
\label{count_classes}
    Let $J\in\Mon$ and $r\in\A$ with $\deg r<\deg J$. For each $n\ge0$, the
    number of $f\in\Mon$ of degree $n$ with $f\equiv r\pmod{J}$ is exactly
    $q^{n-\deg J}$ if $n\ge\deg J$, and at most $1$ if $n<\deg J$. Consequently
    \[
    \#\{f\in\Mon:\deg f\le N,\ f\equiv r \!\!\pmod{J}\}
    =\frac{\mathcal{Q}_N}{|J|}+\vartheta,\qquad |\vartheta|\le1 .
    \]
\end{lemma}

\begin{proof}
    If $f\equiv r\pmod J$ and $\deg f\ge\deg J$, then $f=r+gJ$ with $g\ne0$
    and $\deg f=\deg g+\deg J$; since $J$ is monic, $f$ is monic exactly when
    $g$ is. This gives $q^{n-\deg J}$ polynomials of each degree
    $n\ge\deg J$. A polynomial $f\equiv r$ of degree $<\deg J$ must equal $r$,
    so there is at most one of those. If $\deg J\le N$ the total is therefore
    $\ind_{r\in\Mon}+\sum_{j=0}^{N-\deg J}q^j
    =\ind_{r\in\Mon}+\frac{q^{N-\deg J+1}-1}{q-1}$, and
    \[
    \frac{q^{N-\deg J+1}-1}{q-1}-\frac{\mathcal{Q}_N}{|J|}
    =\frac{q^{-\deg J}-1}{q-1}\in(-1,0],
    \]
    so $|\vartheta|\le1$. If $\deg J>N$ the total is $0$ or $1$, while
    $0<\frac{\mathcal{Q}_N}{|J|}\le\frac{q^{N+1}-1}{(q-1)q^{N+1}}<1$; again
    $|\vartheta|\le1$.
\end{proof}
\begin{lemma}
\label{siftable_1}
    Let $a,h \in \Mon$, $c\in\mathbb{A}$ be such that $|c|\le |h|$ and $(h,c),(h,c+a)$ are both
    squarefree. Then
    \begin{multline*}
        \sum_{\substack{\deg f\le N \\ f\in \Mon \\ f\equiv c \pmod{h} }}
        \mu^2(f)\mu^2(f+a)
        = \frac{E_h(N)}{|h|}\cdot
        \frac{\Eul((h,c))\,\Eul((h,c+a))}{|(h,c)|\,|(h,c+a)|} \\
        \times\prod_{\substack{p^2\| a \\ p\mid h\implies p\mid c+a}}
        \left(1+\frac{(1-\ind_{p\mid h}/|p|)^{-1}}{|p|^2 -1} \right)
        + O_{a,c,h,q}\!\left(N^3q^{\frac{2N}{3}}\right)
    \end{multline*}
    
\end{lemma}

\begin{proof}
Recall that for any $g\in \Mon$ we have
$\mu^2(g)=\sum_{e^2\mid g}\mu(e)$. Thus
\begin{align*}
    E(N;a,c,h):&= \sum_{\substack{f\in\Mon \\ \deg f\le N \\
    f\equiv c \pmod{h}}} \mu^2(f)\mu^2(f+a)
    = \sum_{\substack{f\in\Mon \\ \deg f\le N \\
    f\equiv c \pmod{h}}} \ \sum_{e_1^2\mid f,\ e_2^2\mid f+a}\mu(e_1)\mu(e_2) \\
    &= \sum_{|e_1|,|e_2|\le q^{\frac{N}{2}}}\mu(e_1)\mu(e_2)N_a(N;e_1,e_2,c,h),
\end{align*}
where
\[
N_a(N;e_1,e_2,c,h) := \#\{ f\in \Mon \mid \deg f \le N,\ e_1^2\mid f,\
e_2^2\mid f+a,\ f\equiv c \pmod{h}\}.
\]

Let $1\le L\le \frac{N}{2}$ be a parameter to be chosen later. Then
\begin{align*}
     E(N;a,c,h) & = \sum_{|e_1e_2|\le q^L}\mu(e_1)\mu(e_2)N_a(N;e_1,e_2,c,h)
     + \sum_{|e_1e_2|> q^L}\mu(e_1)\mu(e_2)N_a(N;e_1,e_2,c,h) \\
     &=:T_1 + T_2 \quad \text{(respectively).}
\end{align*}
First note that $N_a(N;e_1,e_2,c,h)$ counts the simultaneous solutions to the system of congruences $f\equiv c \pmod{h}$, $f\equiv 0 \pmod{e_1^2}$,
$f\equiv -a \pmod{e_2^2}$. If $J=[h,e_1^2,e_2^2]$, then by the Chinese
Remainder Theorem a solution exists if and only if
\begin{equation}
(h,e_1^2)\mid c, \qquad (h,e_2^2)\mid c+a, \qquad (e_1^2,e_2^2)\mid a .
\end{equation}
When these conditions hold the solutions form a single residue class modulo
$J$. Hence by \cref{count_classes}
\begin{equation}
\label{N_a_count}
N_a(N;e_1,e_2,c,h) = \frac{\mathcal{Q}_N}{|[h,e_1^2,e_2^2]|}
\ind_{(h,e_1^2)\mid c}\,\ind_{(h,e_2^2)\mid c+a}\,\ind_{(e_1^2,e_2^2)\mid a}
+\vartheta,\qquad |\vartheta|\le1,
\end{equation}
But these errors $\vartheta$ only contribute
\begin{align*}
  \left|\sum_{|e_1e_2|\le q^L}\mu(e_1)\mu(e_2)\vartheta \right|& \le \sum_{|e_1e_2|\le q^L}1=\sum_{j\le L}\sum_{\deg(g)=j}\tau(g)\\
  &=\sum_{j\le L}(j+1)q^j\ll Lq^L
\end{align*}
to $T_1$.
Now
\[
[h,e_1^2,e_2^2] = \frac{h[e_1^2,e_2^2]}{(h,[e_1^2,e_2^2])}
=\frac{h}{\left(h,\left(\frac{e_1e_2}{(e_1,e_2)} \right)^2\right)}
\left(\frac{e_1e_2}{(e_1,e_2)} \right)^2 ,
\]
so
\[
T_1= \frac{\mathcal{Q}_N}{|h|}\sum_{\substack{|e_1e_2| \le q^L \\
(h,e_1^2)\mid c,\ (h,e_2^2)\mid c+a \\ (e_1^2,e_2^2)\mid a }}
\frac{\mu(e_1)\mu(e_2)}{|e_1e_2|^2}
\left|(e_1,e_2)^2\left(h,\left(\frac{e_1e_2}{(e_1,e_2)} \right)^2\right)\right|+O(Lq^L) 
\]
Next we split the sum according to the size of
\begin{equation*}
\label{condition_star}
\delta=(e_1,e_2), \qquad d_1=(h,e_1^2), \qquad d_2=(h,(e_2/\delta)^2). \tag{$*$}
\end{equation*}

Note that $d_1d_2=\left(h,\left(\frac{e_1e_2}{(e_1,e_2)}\right)^2\right)$.
Furthermore, if $\mu(e_2)\ne 0$ then $(e_2/\delta,\delta)=1$, so
$(h,e_2^2)=(h,(e_2/\delta)^2)(h,\delta ^2)=d_2(h,\delta^2)$. Hence
$(h,e_2^2)\mid c+a$ if and only if $(h,\delta^2)\mid c+a$ and
$d_2\mid \frac{c+a}{(h,\delta ^2)}$. Therefore
\begin{align*}
    T_1 &= \frac{\mathcal{Q}_N}{|h|}\sum_{\substack{\delta^2\mid a \\ (h,\delta^2)\mid c+a}}
    |\delta|^2\sum_{d_1\mid(h,c)}\ \sum_{d_2\mid\frac{(h,c+a)}{(h,\delta^2)}}|d_1d_2|
    \sum_{\substack{|e_1e_2|\le q^L\\ \text{(\ref{condition_star})} }}
    \frac{\mu(e_1)\mu(e_2)}{|e_1e_2|^2} + O(Lq^L) \\
    &=  \frac{\mathcal{Q}_N}{|h|}\sum_{\substack{\delta^2\mid a \\ (h,\delta^2)\mid c+a}}
    |\delta|^2\sum_{d_1\mid(h,c)}\ \sum_{d_2\mid\frac{(h,c+a)}{(h,\delta^2)}}|d_1d_2|
    \sum_{\text{(\ref{condition_star})} }\frac{\mu(e_1)\mu(e_2)}{|e_1e_2|^2}
    + O_{a,c,h}\!\left(Lq^{N-L}\right) + O(Lq^L)\\
    & =: \mathcal{I} + \mathcal{E},
\end{align*}
where we obtained $O_{a,c,h}\!\left(Lq^{N-L}\right)$ in $\mathcal{E}:=O_{a,c,h}(Lq^{N-L}) + O(Lq^L)$ by estimating the tail of
the series in $\mathcal{I}$ via
\begin{align*}
    &\frac{\mathcal{Q}_N}{|h|}\sum_{\substack{\delta^2\mid a \\ (h,\delta^2)\mid c+a}}
    |\delta|^2\sum_{d_1\mid(h,c)}\ \sum_{d_2\mid\frac{(h,c+a)}{(h,\delta^2)}}|d_1d_2|
    \sum_{\substack{|e_1e_2|> q^L\\ \text{(\ref{condition_star})} }}
    \frac{\mu(e_1)\mu(e_2)}{|e_1e_2|^2} \\
    &\qquad \ll_{c,a,h} q^N\sum_{\substack{|g|>q^L \\ g\in\Mon}}
    \frac{\tau(g)}{|g|^2} \qquad (\text{taking } e_1e_2=g) \\
    &\qquad =q^N\sum_{m>L}\ \sum_{\substack{\deg g=m \\ g\in\Mon}}
    \frac{\tau(g)}{q^{2m}}
    \le q^N\sum_{m>L}\ \frac{(m+1)q^m}{q^{2m}}
    \ll q^N\sum_{m>L}mq^{-m} \ll Lq^{N-L}.
\end{align*}

We now estimate $\mathcal{I}$. To do so we reparametrise the sum defining it.
Fix $d_1,d_2,e_1,e_2,\delta$. Note that $(h,c)$ is squarefree and
$d_1=(e_1^2,h)\mid(h,c)$, so $d_1$ is squarefree, i.e.\ $d_1=(e_1,h)$. Also, if
$\widetilde{e}_2 :=e_2/\delta$, then
$(\widetilde{e}_2^{\,2},h)=d_2\mid\frac{(h,c+a)}{(h,\delta^2)}$, so $d_2$ is
squarefree, i.e.\ $d_2=(\widetilde{e}_2,h)$. Thus write $e_1=d_1m_1$ and
$e_2=d_2m_2$. Note that $(m_1m_2,h)=1$: indeed
$(e_1,\widetilde{e}_2)=1\implies (d_1,d_2)=1$, and
$(e_1/d_1,h)=1=(\widetilde{e}_2/d_2,h)$
, whence
\[
1=\left(\tfrac{e_1\widetilde{e}_2}{d_1d_2},h\right)
 =\left(\tfrac{e_1e_2}{d_1d_2},h\right)=(m_1m_2,h)
\]
(since $e_1\widetilde{e}_2$ and $e_1e_2$ have the same prime factors and
$(d_1,d_2)=1$). Furthermore,
$\delta=(e_1,e_2)=(e_1,d_2m_2)=(e_1,m_2)$ (since $d_2\mid\widetilde{e}_2$,
which is coprime to $e_1$) $=(d_1m_1,m_2)$. But $d_1=(h,e_1)\mid c$,
$(h,e_2)\mid c+a$ and $(c,c+a)=1$; thus $(d_1,\delta)=1$, which means
$\delta=(m_1,m_2)$. Finally, note that $(m_2/\delta,\delta)=(\delta,d_2)=1$ for
all $m_2$ contributing non-trivially to $\mathcal{I}$, as otherwise we would
have $\mu(e_2)=\mu(\delta d_2 (m_2/\delta))=0$. As such,
\begin{align*}
    \mathcal{I}&=\frac{\mathcal{Q}_N}{|h|}
    \sum_{\substack{\delta^2\mid a \\ (h,\delta^2)\mid c+a}}|\delta|^2
    \sum_{d_1\mid(h,c)}\ \sum_{d_2\mid\frac{(h,c+a)}{(h,\delta^2)}}|d_1d_2|
    \sum_{\substack{(m_1,h)=1=(m_2/\delta,h\delta) \\ (m_1,m_2)=\delta} }
    \frac{\mu(d_1m_1)\mu(d_2m_2)}{|d_1m_1d_2m_2|^2} \\
    &=\frac{\mathcal{Q}_N}{|h|}\sum_{\substack{\delta^2\mid a \\ (h,\delta^2)\mid c+a}}
    |\delta|^2\sum_{d_1\mid(h,c)}\frac{\mu(d_1)}{|d_1|}
    \sum_{d_2\mid\frac{(h,c+a)}{(h,\delta^2)}}
    \frac{\mu(d_2)}{|d_2|}
    \sum_{\substack{(m_1,h)=1=(m_2/\delta,h\delta) \\ (m_1,m_2)=\delta} }
    \frac{\mu(m_1)\mu(m_2)}{|m_1m_2|^2}.
\end{align*}

Call the innermost sum $\Sigma_\delta$. Since $(m_1,m_2)=\delta$, write
$m_1=n_1\delta$ and $m_2=n_2\delta$. Then
\begin{align*}
   \Sigma_\delta & := \sum_{\substack{(m_1,h)=1=(m_2/\delta,h\delta) \\
   (m_1,m_2)=\delta} }\frac{\mu(m_1)\mu(m_2)}{|m_1m_2|^2} \\
    &= \frac{\mu(\delta)^2}{|\delta|^4}\sum_{(n_1,h)=1}\frac{\mu(n_1)}{|n_1|^2}
    \sum_{(n_2,h\delta n_1)=1}\frac{\mu(n_2)}{|n_2|^2} \\
    &= \frac{\mu(\delta)^2}{|\delta|^4}\sum_{(n_1,h)=1}\frac{\mu(n_1)}{|n_1|^2}
    \prod_{p\nmid n_1h\delta}\left(1-\frac{1}{|p|^2}\right) \\
    &= \frac{\mu(\delta)^2}{|\delta|^4}
    \prod_{p\mid\delta}\left(1-\frac{1}{|p|^2}\right)^{-1}
    \prod_{p\nmid h}\left(1-\frac{1}{|p|^2}\right)
    \sum_{(n_1,h)=1}\frac{\mu(n_1)}{|n_1|^2}
    \prod_{p\mid n_1}\left(1-\frac{1}{|p|^2}\right)^{-1} \\
    &= \frac{\mu(\delta)^2}{|\delta|^4}
    \prod_{p\mid\delta}\left(1-\frac{1}{|p|^2}\right)^{-1}
    \prod_{p\nmid h}\left(1-\frac{2}{|p|^2}\right).
\end{align*}
Inserting this back into $\mathcal{I}$ we get
\begin{align*}
    \mathcal{I} &= \frac{\mathcal{Q}_N}{|h|}
    \sum_{\substack{\delta^2\mid a \\ (h,\delta^2)\mid c+a}}|\delta|^2
    \sum_{d_1\mid(h,c)}\frac{\mu(d_1)}{|d_1|}
    \sum_{d_2\mid\frac{(h,c+a)}{(h,\delta^2)}}\frac{\mu(d_2)}{|d_2|}
    \cdot\frac{\mu(\delta)^2}{|\delta|^4}
    \prod_{p\mid\delta}\left(1-\frac{1}{|p|^2}\right)^{-1}
    \prod_{p\nmid h}\left(1-\frac{2}{|p|^2}\right) \\
&=\frac{\mathcal{Q}_N}{|h|}\prod_{p\nmid h}\left(1-\frac{2}{|p|^2}\right)
\sum_{\substack{\delta^2\mid a \\ (h,\delta^2)\mid c+a}}
\frac{\mu(\delta)^2}{|\delta|^2}
\prod_{p\mid\delta}\left(1-\frac{1}{|p|^2}\right)^{-1}
\sum_{d_1\mid(h,c)}\frac{\mu(d_1)}{|d_1|}
\sum_{d_2\mid\frac{(h,c+a)}{(h,\delta^2)}}\frac{\mu(d_2)}{|d_2|} \\
&= \frac{E_h(N)}{|h|}\sum_{\substack{\delta^2\mid a \\ (h,\delta^2)\mid c+a}}
\frac{\mu(\delta)^2}{|\delta|^2}
\prod_{p\mid\delta}\left(1-\frac{1}{|p|^2}\right)^{-1}
\prod_{p\mid(h,c)}\left(1-\frac{1}{|p|}\right)
\prod_{p\mid\frac{(h,c+a)}{(h,\delta^2)}}\left(1-\frac{1}{|p|}\right) \\
&= \frac{E_h(N)}{|h|}\cdot
\frac{\Eul((h,c))\,\Eul((h,c+a))}{|(h,c)|\,|(h,c+a)|}
\sum_{\substack{\delta^2\mid a \\ (h,\delta^2)\mid c+a}}
\frac{\mu(\delta)^2}{|\delta|^2}
\prod_{p\mid\delta}\left(1-\frac{1}{|p|^2}\right)^{-1}
\prod_{p\mid\delta}\left(1-\frac{\ind_{p\mid h}}{|p|}\right)^{-1} \\
&= \frac{E_h(N)}{|h|}\cdot
\frac{\Eul((h,c))\,\Eul((h,c+a))}{|(h,c)|\,|(h,c+a)|}
\prod_{\substack{p^2\| a \\ p\mid h\implies p\mid(h,c+a)}}
\left[1+\frac{\left(1-\frac{\ind_{p\mid h}}{|p|}\right)^{-1}}{|p|^2-1}\right].
\end{align*}

Thus $T_1$ becomes
\begin{equation}
    T_1= \frac{E_h(N)}{|h|}\cdot
    \frac{\Eul((h,c))\,\Eul((h,c+a))}{|(h,c)|\,|(h,c+a)|}
    \!\!\!\!\!\!\!\!\prod_{\substack{p^2\| a \\ p\mid h\implies p\mid(h,c+a)}}
    \left[1+\frac{\left(1-\frac{\ind_{p\mid h}}{|p|}\right)^{-1}}{|p|^2-1}\right]
    + O_{a,c,h}\!\left(Lq^{N-L}+Lq^L\right) \label{T_1}
\end{equation}

We now start estimating $T_2$. For each $k,l\le N/2$ define
\[
M_{k,l}(N) := \sum _{\substack{|e_1|=q^k \\ |e_2| =q^l}} N_a(N;e_1,e_2,c,h).
\]
Then
\begin{equation}
  \begin{aligned}
    \label{T_2ineq}
    T_2 &= \sum_{\substack{ |e_1|,|e_2|\le q^{N/2} \\ |e_1e_2|> q^L}}
    \mu(e_1)\mu(e_2)N_a(N;e_1,e_2,c,h)  \\
    & \ll \sum_{\substack{k,l\le N/2 \\ k+l >L}}
    \sum_{\substack{|e_1|=q^k \\ |e_2|=q^l}}N_a(N;e_1,e_2,c,h)
    \ \le \ M_{k',l'}(N)N^2 ,
  \end{aligned}
\end{equation}
where $(k',l')$ is the pair that maximises $M_{k,l}$. Assume without loss of
generality that $k'\ge l'$. It follows from the definition that
 \[
 N_a(N;e_1,e_2,c,h) \le \sum_{\deg f\le N}
 \ind_{f\equiv 0\!\!\pmod{e_1^2}}\ \ind_{f+a\equiv 0\!\!\pmod{e_2^2}}.
 \]
 Writing $f=ge_1^2$ we have $|g|\le |f|/|e_1^2|=q^{N-2k'}$; therefore
 $\deg g\le N-2k'$ and $ge_1^2\equiv-a\pmod{e_2^2}$. Thus
 \[
 M_{k',l'}(N)\le \sum_{\substack{\deg e_2=l' \\ \deg g \le N-2k' }}
 \ \sum_{\substack{\deg e_1=k' \\ ge_1^2\equiv-a\!\!\pmod{e_2^2}}} 1 ,
 \]
 where the inner sum can be split in terms of residue classes mod $e_2^2$:
 \begin{align*}
 \sum_{\substack{\deg e_1=k' \\ ge_1^2\equiv-a\!\!\pmod{e_2^2}}} 1
 &= \sum_{\substack{|b|<|e_2|^2 \\ gb^2\equiv-a\!\!\pmod{e_2^2}}}
 \ \sum_{\substack{\deg e_1=k' \\ e_1\equiv b\!\!\pmod{e_2^2}}} 1
 = \sum_{\substack{|b|<|e_2|^2 \\ gb^2\equiv-a\!\!\pmod{e_2^2}}}
 \max\{q^{k'-2l'},1\}\\
 &= \nu_{g,e_2}\max\{q^{k'-2l'},1\}\le \left(1+q^{k'-2l'}\right)\nu_{g,e_2},
 \end{align*}
 
 where for each $e_2$ and $g$ with $\deg e_2=l'$ and $\deg g\le N-2k'$,
 $\nu_{g,e_2}$ denotes the number of solutions
 modulo $e_2^2$ to the
 congruence
 \[
 g y^2 +a \equiv 0 \pmod{e_2^2}.
 \]
 By the Chinese Remainder Theorem we have
 $\nu_{g,e_2}=\prod_{p\mid e_2}\nu_{g,p}$. Now for each $p\mid e_2$ with
 $p\nmid g$, the congruence $gy^2 +a \equiv 0 \pmod {p}$ has at most $2$
 solutions, and any one such solution extends uniquely to a solution of
 $gy^2+a \equiv0 \pmod{p^2}$ by Hensel's Lemma. Thus $\nu_{g,p}\le 2$ for all
 $p\mid e_2$ with $p\nmid g$. On the other hand, if $p\mid g$ then
 $\nu_{g,p}=0$ unless $p\mid a$, which is possible for at most
 $\omega(a)=O_a(1)$ primes $p$. For each such $p$, any $y$ modulo $p$ is a
 solution to $gy^2+a \equiv0 \pmod{p^2}$; in particular, for any
 $p\mid (a,g,e_2)$ we have $\nu_{g,p}\le |a|^2$. Therefore
 \[
 \nu_{g,e_2}\le |a|^{2\omega((a,g,e_2))}\prod_{\substack{p\mid e_2 \\ p\nmid g}}
 \nu_{g,p} \ \ll_a \prod_{\substack{p\mid e_2 \\ p\nmid g}}\nu_{g,p}
 \le 2^{\omega(e_2)} \le \tau (e_2)
 \]
 uniformly for all $|g|\le q^{N-2k'}$. Consequently
 \begin{align*}
     M_{k',l'}(N) &\le \sum_{\substack{\deg e_2=l' \\ \deg g \le N-2k' }}\!\!\!\!\!
     \left(1+q^{k'-2l'}\right)\nu_{g,e_2}
     \ \ll \  \left(1+q^{k'-2l'}\right)\sum_{\deg e_2=l'}\tau(e_2) \!\!\!\!\sum_{\deg g\le N-2k'}\!\!\!\!1 \\
     & \ll \left(1+q^{k'-2l'}\right) q^{N-2k'}\sum_{\deg e_2=l'}\tau(e_2)
     \ = \ (l'+1)\left( q^{N-2k'+l'}+q^{N-(k'+l')}\right) \\
     & \le N\left(q^{N-k'}+q^{N-L}\right) \qquad \text{[since $k'+l'>L$].}\\
     &\le N\left(q^{N-\frac{L}{2}}+q^{N-L}\right) \qquad \text{[since $2k'\ge k'+l'>L$].}\\
     &\ll Nq^{N-\frac{L}{2}}
 \end{align*}
Thus
 \[
 T_2 \le M_{k',l'}(N)N^2\ll_{a,q} N^3q^{N-\frac{L}{2}}.
 \]
 Thus, combining our estimates
 for $T_1$ and $T_2$, we get
 \begin{multline*}
        \sum_{\substack{\deg f\le N \\ f\in \Mon \\ f\equiv c \pmod{h} }}
        \!\!\!\!\!\mu^2(f)\mu^2(f+a)= \frac{E_h(N)}{|h|}\cdot
        \frac{\Eul((h,c))\,\Eul((h,c+a))}{|(h,c)|\,|(h,c+a)|} \\
        \times\prod_{\substack{p^2\| a \\ p\mid h\implies p\mid c+a}}
        \left(1+\frac{(1-\ind_{p\mid h}/|p|)^{-1}}{|p|^2 -1} \right)
        + O_{a,c,h,q}\!\left(Lq^L+Lq^{N-L}+N^3q^{N-\frac{L}{2}}\right).
\end{multline*}
Taking $L=[2N/3]$ we get $Lq^L+Lq^{N-L}+N^3q^{N-\frac{L}{2}} \ll_q N^3q^{\frac{2N}{3}}$, which
gives the result.
\end{proof}

\begin{theorem}
\label{siftable_thm}

    Let $k_1,k_2 \in \Mon$ be such that $k_1,k_2$ are squarefree and coprime
    to~$a$. Then
    \begin{multline}
         \sum_{\substack{\deg f\le N \\ f\in \Mon \\ f\equiv 0 \pmod{k_1} \\
         f+a\equiv 0\pmod{k_2} }}\!\!\!\!\!\!\!\!\mu^2(f)\mu^2(f+a)
         = \frac{\mathcal{Q}_N}{|[k_1,k_2]|}\cdot\frac{\Eul([k_1,k_2])}{|[k_1,k_2]|}
         \prod_{p\nmid [k_1,k_2]}\left(1-\frac{2}{|p|^2}\right)
         \prod_{p^2\| a}\left(1+\frac{1}{|p|^2-1}\right) \\
         + O\!\left(N^3q^{\frac{2N}{3}}\right).
    \end{multline}
Consequently, if $(k_1,k_2)=1$ then
\begin{multline}
\label{siftable_3}
    \sum_{\substack{\deg f\le N \\ f\in \Mon \\ f\equiv 0 \pmod{k_1} \\
    f+a\equiv 0\pmod{k_2} }}\!\!\!\!\!\!\!\!\mu^2(f)\mu^2(f+a)\ind_{(f,a)=1}
    = E(N;a) \frac{\Eul(k_1)\Eul(k_2)}{|k_1k_2|^2}
    \prod_{p \mid k_1k_2}\left(1-\frac{2}{|p|^2}\right)^{-1} \\
    +O\!\left(N^3q^{\frac{2N}{3}}\right),
\end{multline}
where
$E(N;a) := E(N)\prod_{p\mid a}\left(1-\frac{|p|-1}{|p|^2-2}\right)
\prod_{p^2\| a}\left(1+\frac{1}{|p|^2-1}\right)$.
\end{theorem}

\begin{proof}
    By the Chinese Remainder Theorem the simultaneous congruences
    $f\equiv0 \pmod{k_1}$ and $f\equiv-a \pmod{k_2}$ are equivalent to
    $f \equiv c \pmod{[k_1,k_2]}$ for a suitable $c$ modulo $[k_1,k_2]$
    depending on~$a$. Thus by \cref{siftable_1},
 \begin{align}
 \label{siftable_2}
      \sum_{\substack{\deg f\le N \\ f\in \Mon \\ f\equiv 0 \pmod{k_1} \\
      f+a\equiv 0\pmod{k_2} }}\mu^2(f)\mu^2(f+a)
      &=  \sum_{\substack{\deg f\le N \\ f\in \Mon \\
      f\equiv c \pmod{[k_1,k_2]} }}\mu^2(f)\mu^2(f+a) \notag \\
      &= \frac{E_{[k_1,k_2]}(N)}{|[k_1,k_2]|}\cdot
      \frac{\Eul(([k_1,k_2],c))\,\Eul(([k_1,k_2],c+a))}
           {|([k_1,k_2],c)|\,|([k_1,k_2],c+a)|} \notag \\
      &\qquad \times \prod_{p^2\| a}
      \left(1+\frac{(1-\ind_{p\mid [k_1,k_2]}/|p|)^{-1}}{|p|^2 -1} \right) \notag \\
      &\qquad\qquad + O\!\left(N^3q^{\frac{2N}{3}}\right),
 \end{align}
 
 where the condition $p\mid[k_1,k_2]\implies p\mid c+a$ is omitted because
 $([k_1,k_2],a)=1$. Note that we can write $c\equiv -amk_1 \pmod{[k_1,k_2]}$,
 where $m$ is a residue class modulo $k_2/(k_1,k_2)$ such that
 $mk_1\equiv 1 \pmod{k_2/(k_1,k_2)}$. Therefore
 \begin{align*}
 ([k_1,k_2],c)&= \left( \frac{k_1k_2}{(k_1,k_2)},amk_1 \right)
 = k_1\left( \frac{k_2}{(k_1,k_2)},am \right)\\
 &=k_1\left( \frac{k_2}{(k_1,k_2)},a \right)
 \qquad \text{[since $(m,k_2/(k_1,k_2))=1$]} \\
 &=k_1 \qquad \text{[since $(a,k_2)=1$].}
  \end{align*}
 Similarly,
 \begin{align*}
     c+a&= h[k_1,k_2] +a(1-mk_1) \\
        &= h[k_1,k_2] +ag\frac{k_2}{(k_1,k_2)}
        =\frac{k_2}{(k_1,k_2)}(hk_1 + ag),
 \end{align*}
 where $1-mk_1=gk_2/(k_1,k_2)$, which also implies that $(k_1,g)=1$. Moreover,
 \begin{align*}
 \left([k_1,k_2],c+a\right) &= \frac{k_2}{(k_1,k_2)}\left(k_1, hk_1+ag\right)
 =\frac{k_2}{(k_1,k_2)}\left(k_1,ag\right)\\
 &=\frac{k_2}{(k_1,k_2)} \qquad \text{[since $(k_1,g)=1=(k_1,a)$].}
 \end{align*}
 Thus we can rewrite \Cref{siftable_2} as
 \begin{align*}
     &\frac{\mathcal{Q}_N}{|[k_1,k_2]|}\cdot
     \frac{\Eul(k_1)\Eul\!\left(\frac{k_2}{(k_1,k_2)}\right)}
          {\left|k_1\frac{k_2}{(k_1,k_2)}\right|}
     \prod_{p\nmid [k_1,k_2]}\left(1-\frac{2}{|p|^2}\right)
     \prod_{p^2\| a}\left(1+\frac{1}{|p|^2-1}\right)
     + O\!\left(N^3q^{\frac{2N}{3}}\right) \\
     &\qquad = \mathcal{Q}_N\frac{\Eul([k_1,k_2])}{|[k_1,k_2]|^2}
     \prod_{p\nmid [k_1,k_2]}\left(1-\frac{2}{|p|^2}\right)
     \prod_{p^2\| a}\left(1+\frac{1}{|p|^2-1}\right)
     + O\!\left(N^3q^{\frac{2N}{3}}\right),
 \end{align*}
 which settles the first part of the theorem.

 Now assume that $k_1$ and $k_2$ are coprime. Then
 \begin{align*}
    &  \sum_{\substack{\deg f\le N \\ f\in \Mon \\ f\equiv 0 \pmod{k_1} \\
    f+a\equiv 0\pmod{k_2} }}\mu^2(f)\mu^2(f+a)\ind_{(f,a)=1}
    =  \sum_{\substack{\deg f\le N \\ f\in \Mon \\ f\equiv 0 \pmod{k_1} \\
    f+a\equiv 0\pmod{k_2} }}\mu^2(f)\mu^2(f+a)\sum_{e\mid(f,a)}\mu(e) \\
      &\qquad = \sum_{e\mid a}\mu(e)\!\!\!\!\!\sum_{\substack{\deg f\le N \\ f\in \Mon \\
      f\equiv 0 \pmod{ek_1} \\ f+a\equiv 0\pmod{k_2} }}\mu^2(f)\mu^2(f+a) \\
      &\qquad = \sum_{e\mid a}\mu(e) \frac{\mathcal{Q}_N\Eul(ek_1k_2)}{|[ek_1,k_2]|^2}
      \prod_{p\nmid [ek_1,k_2]}\left(1-\frac{2}{|p|^2}\right)
      \prod_{p^2\| a}\left(1+\frac{1}{|p|^2-1}\right)
      + O\!\left(N^3q^{\frac{2N}{3}}\right)\\
      &\qquad =  \left(\sum_{e\mid a}\mu(e)\frac{\Eul(e)}{|e|^2}
      \prod_{p\mid e}\left(1-\frac{2}{|p|^2}\right)^{-1}\right)
      \frac{\mathcal{Q}_N\Eul(k_1k_2)}{|k_1k_2|^2}
      \prod_{p\nmid k_1k_2}\!\!\left(1-\frac{2}{|p|^2}\right)\!\!
      \prod_{p^2\| a}\!\!\left(1+\frac{1}{|p|^2-1}\right) \\
      &\qquad\qquad +  O\!\left(N^3q^{\frac{2N}{3}}\right) \\
      &\qquad = \prod_{p\mid a}\!\!\left(1-\frac{|p|-1}{|p|^2-2}\right)\!
      \frac{\mathcal{Q}_N\Eul(k_1k_2)}{|k_1k_2|^2}\!
      \prod_{p\nmid k_1k_2}\!\!\left(1-\frac{2}{|p|^2}\right)\!\!
      \prod_{p^2\| a}\!\!\left(1+\frac{1}{|p|^2-1}\right)\!\!
      +  O\!\left(N^3q^{\frac{2N}{3}}\right) \\
      &\qquad = \mathcal{Q}_N\!\prod_p \!\!\left(1-\frac{2}{|p|^2}\right)\!\!
      \prod_{p\mid a}\!\!\left(1-\frac{|p|\!-\!1}{|p|^2\!-\!2}\right)\!\!
      \prod_{p^2\| a}\!\!\left(1+\frac{1}{|p|^2\!-\!1}\right)\!
      \frac{\Eul(k_1)\Eul(k_2)}{|k_1k_2|^2}\!
      \prod_{p \mid k_1k_2}\!\!\left(1-\frac{2}{|p|^2}\right)^{-1} \\
      &\qquad\qquad +  O\!\left(N^3q^{\frac{2N}{3}}\right) \\
      &\qquad = E(N;a) \frac{\Eul(k_1)\Eul(k_2)}{|k_1k_2|^2}
      \prod_{p \mid k_1k_2}\left(1-\frac{2}{|p|^2}\right)^{-1}
      + O\!\left(N^3q^{\frac{2N}{3}}\right),
 \end{align*}
 
 where
\begin{align*}
E(N;a) :&= \mathcal{Q}_N\prod_p \left(1-\frac{2}{|p|^2}\right)
\prod_{p\mid a}\left(1-\frac{|p|-1}{|p|^2-2}\right)
\prod_{p^2\| a}\left(1+\frac{1}{|p|^2-1}\right) \\
&= E(N)\prod_{p\mid a}\left(1-\frac{|p|-1}{|p|^2-2}\right)
\prod_{p^2\| a}\left(1+\frac{1}{|p|^2-1}\right). \qedhere
\end{align*}
\end{proof}

Note that if we define $\rho :\Mon\longrightarrow \mathbb{R}$ by
\begin{align*}
\rho(m)&=\frac{\Eul(m)}{|m|^2}\prod_{p\mid m}\left( 1- \frac{2}{|p|^2} \right)^{-1}
= \frac{1}{|m|}\prod_{p\mid m}\left(1-\frac{|p|-2}{|p|^2-2}\right),
\end{align*}
then \Cref{siftable_3} can be rewritten as
\[
\sum_{\substack{\deg f\le N \\ f\in \Mon \\ f\equiv 0 \pmod{k_1} \\
f+a\equiv 0\pmod{k_2} }}\mu^2(f)\mu^2(f+a)\ind_{(f,a)=1}
= E(N;a)\rho(k_1k_2) + O\!\left(N^3q^{\frac{2N}{3}}\right).
\]
We are now ready to prove \cref{siftability}.

\begin{proof}[Proof of \Cref{siftability}]
 Fix $a\in\Mon$. Applying \cref{siftable_thm} with $k_1=k_2=1$, we get
 \begin{align*}
   \sum_{\substack{\deg f\le N \\ f\in \Mon }}\!\mu^2(f)\mu^2(f+a)\ind_{(f,a)=1}
   &= E(N;a) \\
   &= \mathcal{Q}_N\!\prod_p \!\left(1-\frac{2}{|p|^2}\right)\!
   \prod_{p\mid a}\!\left(1-\frac{|p|-1}{|p|^2-2}\right)\!
   \prod_{p^2\| a}\!\left(1+\frac{1}{|p|^2-1}\right)\\
   & \asymp q^N,
 \end{align*}
Thus condition~(i) of
siftability holds. Moreover, for any coprime $k_1,k_2$ with
$|k_1k_2|\le q^N$ we have
\[
\sum_{\substack{\deg f\le N \\ f\in \Mon \\ f\equiv 0 \pmod{k_1} \\
f+a\equiv 0\pmod{k_2} }}\mu^2(f)\mu^2(f+a)\ind_{(f,a)=1}
= \rho(k_1k_2)\left[E(N;a) + R_a(N;k_1,k_2)\right],
\]
where $R_a(N;k_1,k_2):= \rho(k_1k_2)^{-1}O(N^3q^{\frac{2N}{3}})$. Clearly
$0\le \rho(p)\le\frac{1}{|p|}$ for every prime~$p$. Now note that for any fixed
$\theta\in(0,1/3)$
\begin{align}
    \sum_{\substack{|k_1k_2|\le q^{N\theta} \\ \mu^2(k_1k_2)=1}}
    \rho(k_1k_2)R_a(N;k_1,k_2)
    & \ll N^3q^{\frac{2N}{3}}\sum_{\deg (k_1k_2)\le N\theta}1 = N^3q^{\frac{2N}{3}}\sum_{\deg f\le N\theta}\tau(f) \notag \\
    &= N^3q^{\frac{2N}{3}}\sum_{m=1}^{[N\theta]}\ \sum_{\deg f=m}\tau(f)
    \le N^3q^{\frac{2N}{3}}\sum_{m=1}^{[N\theta]}mq^m \notag \\
    &\ll N^3q^{\frac{2N}{3}+N\theta} \notag \\
    & = q^{N(1-\epsilon)}\quad 
    \ \ll \ \frac{q^N}{N^3}. \notag
\end{align}
For some $\epsilon >0$ .Thus $S=\{f\in \Mon \mid \mu^2(f)=1\}$ is siftable with respect to~$a$.
Furthermore, since $\rho$ and $\theta$ do not depend on~$a$, $S$ is indeed
\emph{siftable}.

Lastly, using \cref{Mertens_thm}
 we see that
\begin{align*}
    \lambda_\rho(M)&=\sum_{\deg p\le M} \rho(p)
= \sum_{\deg p \le M}\frac{1}{|p|}\left(1-\frac{|p|-2}{|p|^2-2}\right) \\
&= \sum_{\deg p \le M}\frac{1}{|p|} + O(1) \asymp \log M .
\end{align*}

Thus $S$ is regular.
\end{proof}

\section{The probability space of
$(\widetilde{\omega}_{M},\widetilde{\omega}_{M}\circ T_a)$:
the bivariate Kubilius model for polynomials}
\label{sec:Mass}

\begin{center}\rule{0.45\textwidth}{0.4pt}\end{center}

\noindent\textbf{Standing conventions for this section and the next.}
Our
$
\theta\in\left(0,\tfrac13\right)
$(the level of distribution) and
$\Delta>0$ are arbitrary but fixed and we let $C=\max\{C_0,\frac{2\Delta}{\theta\log q}\}$ where $C_0$ is the constant supplied by
\cref{brun_quantitative}. All results below assume
$ M\ \le\ \frac{N}{C\log N},
$
which guarantees $\frac{N}{M}\ge C_0\log M$
and hence the sieve error $M^{-\Delta}$ of \cref{sieve_main_thm}.

\begin{center}\rule{0.45\textwidth}{0.4pt}\end{center}

We are now ready to study the measure associated to
$(\widetilde{\omega}_{M},\widetilde{\omega}_{M}\circ T_a)$. We restrict
ourselves to the subset of $\Mon$ that is of interest to us, namely
\[
\Fset(N;a):= \{f\in \Mon \mid \deg f\le N,\ f\in S,\ f+a\in S,\ (f,a)=1 \}
\subset \Mon .
\]
Now for $k_1,k_2\mid\Pam$ with $(k_1,k_2)=1$ we define
\[
E_{k_1,k_2} :=\{ f\in \Fset(N;a) : k_1\mid f,\ k_2\mid f+a,\
(f,\Pam/k_1)=(f+a,\Pam/k_2)=1\}.
\]
Essentially, $E_{k_1,k_2}$ is the set of all polynomials $f$ of degree at most
$N$ such that $f/k_1$ and $(f+a)/k_2$ are polynomials divisible only by primes
$p$ of degree greater than~$M$.

Note that if $f\in E_{k_1,k_2}$ then $\omega_M(f)=\omega(k_1)$ and
$\omega_M(f+a)=\omega(k_2)$.

Now let $\Ecal:= \{E_{k_1,k_2} : k_1,k_2\mid\Pam,\ (k_1,k_2)=1\}$;
then it
is easy to see that $\Ecal$ is a partition of $\Fset(N;a)$.

Thus if $\mathcal{G}$ = $\sigma$-algebra generated by $\Ecal$, then $(\omega_{M},\omega_{M}\circ T_a): \left(\Fset(N;a),\mathcal{G}\right) \rightarrow \mathbb{R}^2$ is measurable. Furthermore, we define a probability measure
$\nu_M: \mathcal{G} \longrightarrow [0,1]$ by
\[
\nu_M(E_{k_1,k_2})=\frac{|E_{k_1,k_2}|}{|\Fset(N;a)|},
\]
extended additively to all of $\mathcal{G}$.

\begin{proposition}
\label{nu_1}
    Fix $k_1,k_2 \mid\Pam$ with $(k_1,k_2)=1$ and
    $|k_1k_2|\le q^{\frac{N\theta}{2}}$, where $\theta$ is the level of
    distribution of~$S$. Then
    \[
    \nu_M(E_{k_1,k_2})=\left(1+O\!\left(M^{-\Delta}
    +\frac{1}{N}\right)\right)\rho(k_1k_2)
    \prod_{p\mid\Pam/k_1k_2}\left(1-2\rho(p)\right),
    \]
    where $\Delta$ is the saving exponent of \Cref{sieve_main_thm}.
\end{proposition}

To prove this, we first prove the following lemma.

\begin{lemma}
\label{nu_0}
    With the notation above,
    \begin{equation*}
    \nu_M(E_{k_1,k_2})=\left(1+O\!\left(M^{-\Delta}\right)\right)
    \rho(k_1k_2)\prod_{p\mid\Pam/k_1k_2}\left(1-2\rho(p)\right)
    + O\!\left(\frac{R_a^*(N;k_1,k_2)}{|\Fset(N;a)|} \right),
    \end{equation*}
    where
    \[
   R_a^*(N;k_1,k_2) = \sum_{\substack{d_1\mid\Pam/k_1 \\ d_2\mid \Pam/k_2 \\
   (k_1d_1,k_2d_2 )=1 \\ |d_i|\le q^{\theta N/4}  }}
   \rho(k_1d_1k_2d_2) \left| R_a(N;k_1d_1,k_2d_2) \right| .
    \]
    
\end{lemma}

\begin{proof}
For coprime $k_1,k_2\mid \Pam$,
\begin{align*}
    |E_{k_1,k_2}| & = \sum_{\substack{\deg f\le N \\ f\equiv 0 \pmod{k_1} \\
    f\equiv -a \pmod{k_2}}} \ind_{(f,\Pam/k_1)=1}\ind_{(f+a,\Pam/k_2)=1}
    \ind_S(f)\ind_S(f+a)\ind_{(f,a)=1} \\
    & = \sum_{\substack{\deg f\le N \\ f\equiv 0 \pmod{k_1} \\
    f\equiv -a \pmod{k_2}}} \!\!\!\!(1*\mu)\left(f,\Pam/k_1\right)
    (1*\mu)\left(f+a,\Pam/k_2\right)\ind_S(f)\ind_S(f+a)\ind_{(f,a)=1}
\end{align*}
Now, to estimate $|E_{k_1,k_2}|$ we bound $(1*\mu)$ by the sieve weights
discussed in \cref{sieve_main_thm}. Thus for each $d\mid\Pam$,
\[
(1*\gamma^-)(d) \le (1*\mu)(d)= \ind_{d=1} \le  (1*\gamma^+)(d).
\]

Recall from \cref{sieve_main_thm} that the support of $\gamma^-,\gamma^+$ is
$\deg d\le\frac{\theta N}{4}$; we will use this fact later. Now define
\begin{gather}
\label{psi-_1}
\Psi^+(f,g):= (1*\gamma^+)(f)(1*\gamma^+)(g), \notag \\
\Psi^-(f,g) :=  (1*\gamma^+)(f)(1*\gamma^-)(g) + (1*\gamma^-)(f)(1*\gamma^+)(g)
- (1*\gamma^+)(f)(1*\gamma^+)(g).
\end{gather}

Obviously $(1*\mu)(f)(1*\mu)(g) \le \Psi^+(f,g)$ for all $f,g \in \A$. Also,
\begin{align}
\label{psi-_2}
    \Psi^-(f,g) &= (1*\gamma^-)(f)(1*\gamma^-)(g)-
    \left[(1*\gamma^+)(f)-(1*\gamma^-)(f)\right]
    \left[(1*\gamma^+)(g)(1*\gamma^-)(g)\right]  \notag \\
    & \le (1*\gamma^-)(f)(1*\gamma^-)(g).
\end{align}
Thus from \cref{psi-_1},\Cref{psi-_2}, $\Psi^-(f,g) \le (1*\mu)(f)(1*\mu)(g) $ regardless the sign of $(1*\gamma^-)(f)$, $(1*\gamma^-)(g)$.
Thus $\Psi^-(f,g) \le (1*\mu)(f)(1*\mu)(g) \le \Psi^+(f,g)$, and so
\begin{align*}
    |E_{k_1,k_2}| &  \le \sum_{\substack{\deg f\le N \\ f\equiv 0 \pmod{k_1} \\
    f\equiv -a \pmod{k_2}}}\Psi^+\!\left( \left( f, \Pam/k_1 \right),
    \left( f+a,\Pam/k_2 \right) \right)\ind_S(f)\ind_S(f+a)\ind_{(f,a)=1} \\
    & = \sum_{\substack{\deg f\le N \\ f\equiv 0 \pmod{k_1} \\
    f\equiv -a \pmod{k_2}}}\ \sum_{d_1\mid\left( f, \Pam/k_1 \right)}
    \ \sum_{d_2\mid \left( f+a, \Pam/k_2 \right)}
    \gamma^+(d_1)\gamma^+(d_2)\ind_S(f)\ind_S(f+a)\ind_{(f,a)=1} \\
    & =\sum_{\substack{d_1\mid\Pam/k_1 \\ d_2\mid \Pam/k_2 \\
    (k_1d_1,k_2d_2)=1 \\ |d_i| \le q^{\theta N/4} }} \gamma^+(d_1)\gamma^+(d_2)
    \sum_{\substack{\deg f\le N \\ f\equiv 0 \pmod{k_1d_1} \\
    f\equiv -a \pmod{k_2d_2}}} \ind_S(f)\ind_S(f+a)\ind_{(f,a)=1} \\
    & \le |\Fset(N;a)| \sum_{\substack{d_1\mid\Pam/k_1 \\ d_2\mid \Pam/k_2 \\
    (k_1d_1,k_2d_2)=1 \\ |d_i| \le q^{\theta N/4} }}
    \gamma^+(d_1)\gamma^+(d_2)\rho (k_1k_2d_1d_2)
    + R_a^*(N;k_1,k_2).
\end{align*}

Label the resulting bound
\begin{align}
\label{level_set_Ubound}
|E_{k_1,k_2}| \le |\Fset(N;a)| \sum_{\substack{d_1\mid\Pam/k_1 \\
d_2\mid \Pam/k_2 \\ (k_1d_1,k_2d_2)=1 \\ |d_i| \le q^{\theta N/4} }}
\gamma^+(d_1)\gamma^+(d_2)\rho (k_1k_2d_1d_2) + R_a^*(N;k_1,k_2).
\end{align}
Similarly, for a lower bound we have
\[
    |E_{k_1,k_2}| \ge \sum_{\substack{\deg f\le N \\ f\equiv 0 \pmod{k_1} \\
    f\equiv -a \pmod{k_2}}}\Psi^-\!\left( \left( f, \Pam/k_1 \right),
    \left( f+a,\Pam/k_2 \right) \right)\ind_S(f)\ind_S(f+a)\ind_{(f,a)=1},
\]
so that
\begin{multline}
\label{level_set_lbound}
     |E_{k_1,k_2}| \ge |\Fset(N;a)| \!\!\!\!\!\sum_{\substack{d_1\mid\Pam/k_1 \\
     d_2\mid \Pam/k_2 \\ (k_1d_1,k_2d_2)=1 \\ |d_i| \le q^{\theta N/4} }}
     \!\!\!\!\!\!\!\left[\gamma^+(d_1)\gamma^-(d_2) + \gamma^-(d_1)\gamma^+(d_2)
     - \gamma^+(d_1)\gamma^+(d_2)\right]\rho (k_1k_2d_1d_2)  \\
     -3 R_a^*(N;k_1,k_2).
\end{multline}
Now, given $k\in \A$, define $w_k(g):= \rho(g)\ind_{(k,g)=1}$.

Note that for any sign pattern $s_1,s_2 \in \{+,-\}$,

\begin{align*}
    & \sum_{\substack{d_1\mid\Pam/k_1 \\ d_2\mid \Pam/k_2 \\
    (k_1d_1,k_2d_2)=1 \\ |d_i| \le q^{\theta N/4} }}
    \gamma^{s_1}(d_1)\gamma^{s_2}(d_2)\rho (k_1k_2d_1d_2) \\
    & \qquad = \rho(k_1k_2)\sum_{\substack{d_1\mid\Pam/k_1 \\ (d_1,k_2)=1}}
    \gamma^{s_1}(d_1)\rho(d_1)\sum_{\substack{d_2\mid\Pam/k_2 \\
    (d_2,k_1d_1)=1}}\gamma^{s_2}(d_2)\rho(d_2)\\
    & \qquad = \rho(k_1k_2)\sum_{d_1\mid\Pam}\gamma^{s_1}(d_1)w_{k_1k_2}(d_1)
    \sum_{d_2\mid\Pam }\gamma^{s_2}(d_2)w_{k_1k_2d_1}(d_2) \\
    & \qquad = \rho(k_1k_2)\sum_{d_1\mid\Pam}\gamma^{s_1}(d_1)w_{k_1k_2}(d_1)
    \left[ 1+ O\!\left(M^{-\Delta}\right) \right]
    \prod_{p\mid\Pam/k_1k_2d_1}\left(1-w(p)\right) \\
    &\qquad\qquad
    [\text{Using \cref{sieve_main_thm} on the inner sum with } w=w_{k_1k_2d_1}] \\
    & \qquad =\rho(k_1k_2)\left[ 1+ O\!\left(M^{-\Delta}\right)
    \right] \!\!\!\prod_{p\mid\Pam/k_1k_2}\!\!\!\left(1-w(p)\right)
    \sum_{d_1\mid\Pam}\gamma^{s_1}(d_1)w_{k_1k_2}(d_1)
    \prod_{p\mid d_1}\left(1-w(p)\right)^{-1} \\
    & \qquad =\rho(k_1k_2)\left[ 1+ O\!\left(M^{-\Delta}\right)
    \right] \prod_{p\mid\Pam/k_1k_2}(1-\rho(p))
    \left(1-\frac{\rho(p)}{1-\rho(p)}\right) \\
    &\qquad\qquad \left[\text{Using \cref{sieve_main_thm} with }
    w=\frac{w_{k_1k_2}(d)}{\prod_{p\mid d}(1-\rho(p))}\right]
\end{align*}

and hence
\begin{align}
\label{bivariate_sieve}
    \sum_{\substack{d_1\mid\Pam/k_1 \\ d_2\mid \Pam/k_2 \\
    (k_1d_1,k_2d_2)=1 \\ |d_i| \le q^{\theta N/4} }}
    \gamma^{s_1}(d_1)\gamma^{s_2}(d_2)\rho (k_1k_2d_1d_2)
    =\rho(k_1k_2)\left[ 1+ O\!\left(M^{-\Delta}\right) \right]
    \prod_{p\mid\Pam/k_1k_2}(1-2\rho(p)).
\end{align}
Note that in the calculation above, in order to apply \cref{sieve_main_thm} we
need $w_{k_1k_2d_1}$ and $\frac{w_{k_1k_2}(d)}{\prod_{p\mid d}(1-\rho(p))}$ to
satisfy $\Omega_0$ and to be not
equal to~$1$ at primes. These are
clearly satisfied for $w_{k_1k_2d_1}$, since $\rho$ satisfies $\Omega$.
Furthermore, $\frac{w_{k_1k_2}(d)}{\prod_{p\mid d}(1-\rho(p))}$ clearly
satisfies $\Omega_0$ and is not equal to~$1$ as long as $\rho(p)\ne \frac{1}{2}$,
which in our case means $|p|\ne2$; but this is anyway true as long as $q\ne2$.

Now using \Cref{bivariate_sieve} in \Cref{level_set_Ubound} and
\Cref{level_set_lbound} we get
\[
     |E_{k_1,k_2}|\le |\Fset(N;a)|\,\rho(k_1k_2)
     \left[ 1+ O\!\left(M^{-\Delta}\right) \right]
     \prod_{p\mid\frac{\Pam}{k_1k_2}}(1-2\rho(p)) +R_a^*(N;k_1,k_2)
\]
and
\[
     |E_{k_1,k_2}|\ge |\Fset(N;a)|\,\rho(k_1k_2)
     \left[ 1+ O\!\left(M^{-\Delta}\right) \right]
     \prod_{p\mid\frac{\Pam}{k_1k_2}}(1-2\rho(p)) -3R_a^*(N;k_1,k_2).
\]
Thus
\[
     \nu_M(E_{k_1,k_2})=\left[ 1+ O\!\left(M^{-\Delta}\right)
     \right] \rho(k_1k_2)\prod_{p\mid\frac{\Pam}{k_1k_2}}(1-2\rho(p))
     +O\!\left(\frac{R_a^*(N;k_1,k_2)}{|\Fset(N;a)|}\right). \qedhere
\]
\end{proof}

We now turn to the proof of \cref{nu_1}.

\begin{proof}[Proof of \Cref{nu_1}]
 Recall from the proof of \cref{siftability} $\rho(k_1k_2)R_a(N;k_1,k_2)=O(N^3q^\frac{2N}{3})$. Thus for fixed $k_1,k_2$
 \begin{align}
 \label{R_star}
     R_a^*(N;k_1,k_2) & = \sum_{\substack{d_1\mid\Pam/k_1 \\ d_2\mid \Pam/k_2 \\
     (k_1d_1,k_2d_2 )=1 \\ |d_i|\le q^{\theta N/4}  }}\rho(k_1d_1k_2d_2)
     \left| R_a(N;k_1d_1,k_2d_2) \right| \notag \\
     & \ll \sum_{\substack{|d_1|\le q^{\frac{\theta N}{4}} \\
     |d_2|\le q^{\frac{\theta N}{4}} }}N^3q^\frac{2N}{3} \ll N^3q^{N(\frac{\theta}{2}+\frac{2}{3})}
 \end{align}

Also note that
\begin{align*}
    \prod_{p\mid\frac{\Pam}{k_1k_2}}\left(1-2\rho(p)\right)
    \ge \prod_{\deg p \le M}\left(1-\frac{2}{|p|}\right)
    \ \ge \ \exp\left(-2\sum_{|p|\le q^M}\frac{1}{|p|}\right)
    \ \gg \ N^{-2},
\end{align*}

and for $|k_1k_2|\le q^{\frac{N\theta}{2}}$ , $\rho(k_1k_2)\gg q^{-\frac{N\theta}{2}} $
\begin{align*}
    \frac{q^N}{N^3} &\ll \frac{q^N}{N}
    \prod_{p\mid\frac{\Pam}{k_1k_2}}\left(1-2\rho(p)\right)
    \ \ll \ \frac{|\Fset(N;a)|}{N}
    \prod_{p\mid\frac{\Pam}{k_1k_2}}\left(1-2\rho(p)\right),
\end{align*}
Putting this back into \cref{nu_0} gives
\[
\nu_M(E_{k_1,k_2})=\left(1+O\!\left(M^{-\Delta}
+\frac{1}{N}\right)\right)\rho(k_1k_2)
\prod_{p\mid\Pam/k_1k_2}\left(1-2\rho(p)\right). \qedhere
\]
\end{proof}

Next we define $\sigma_M : \mathcal{G} \longrightarrow [0,1]$ as the main term
of $\nu_M(E_{k_1,k_2})$, i.e.
\[
\sigma_M(E_{k_1,k_2})=\rho(k_1k_2)\prod_{p\mid\Pam/k_1k_2}\left(1-2\rho(p)\right).
\]
This is also a probability measure, since
\begin{align}
\label{sigma_is_prob_measure}
    \sigma_M(\Fset(N;a))&= \sum_{\substack{k_1,k_2\mid\Pam \\ (k_1,k_2)=1}}
    \sigma_M(E_{k_1,k_2})
    =\sum_{\substack{k_1,k_2\mid\Pam \\ (k_1,k_2)=1}}
    \rho(k_1k_2)\prod_{p\mid\Pam/k_1k_2}\left(1-2\rho(p)\right) \\  \notag
    &= \prod_{p\mid\Pam}\left(1-2\rho(p)\right)
    \!\!\sum_{k_1\mid\Pam}\!\!\!\rho(k_1)\prod_{p\mid k_1}\left(1-2\rho(p)\right)^{-1}
    \!\!\!\!\!\!\sum_{k_2\mid\Pam/k_1}\!\!\!\!\rho(k_2)\prod_{p\mid k_2}\left(1-2\rho(p)\right)^{-1}\\ \notag
    &=\prod_{p\mid\Pam}\left(1-2\rho(p)\right)
    \sum_{k_1\mid\Pam}\rho(k_1)\prod_{p\mid k_1}\left(1-2\rho(p)\right)^{-1}
    \prod_{p\mid\Pam/k_1} \left(1+ \frac{\rho(p)}{1-2\rho(p)}\right) \\ \notag
    &= \prod_{p\mid\Pam}\left[\left(1-2\rho(p)\right)
    \left(1+ \frac{\rho(p)}{1-2\rho(p)}\right)\right] \\ \notag
    &\qquad\qquad \times \sum_{k_1\mid\Pam}\rho(k_1)\prod_{p\mid k_1}
    \left[\left(1-2\rho(p)\right)
    \left(1+ \frac{\rho(p)}{1-2\rho(p)}\right)\right]^{-1}\\  \notag
    &= \prod_{p\mid\Pam}\left[(1-\rho(p))
    \left(1+ \frac{\rho(p)}{1-\rho(p)}\right)\right] =1 .
\end{align}

Thus for $|k_1k_2|\le q^{\frac{N\theta}{2}}$,
\[
\nu_M(E_{k_1,k_2})=\left(1+O\!\left(M^{-\Delta}
+\frac{1}{N}\right)\right)\sigma_M(E_{k_1,k_2})
= \sigma_M(E_{k_1,k_2}) + O\!\left(R(N,M)\right),
\]
where $R(N,M):=M^{-\Delta}+\frac{1}{N}$. We claim that this
holds for all other $k_1,k_2$ as well, and consequently we will have the
following.

\begin{lemma}
\label{sigma_nu_distance}
    For any $A\in \mathcal{G}$,
    $\nu_M(A)=\sigma_M(A) + O\left(R(N,M)\right)$.
\end{lemma}

\begin{proof}
   For the time being, suppose
    \begin{equation}
    \label{sigma_assumption}
    \sum_{\substack{k_1,k_2\mid \Pam\\ \mu^2(k_1k_2)=1 \\
    |k_1k_2|>q^{\frac{N\theta}{2}}}}\sigma_M(E_{k_1,k_2}) \ll R(N,M).
    \end{equation}
    Then
    \begin{align*}
        \sum_{\substack{k_1,k_2\mid \Pam\\ (k_1,k_2)=1 \\
        |k_1k_2|>q^{\frac{N\theta}{2}}}}\nu_M(E_{k_1,k_2})
        &=1-\sum_{\substack{k_1,k_2\mid \Pam\\ (k_1,k_2)=1 \\
        |k_1k_2|\le q^{\frac{N\theta}{2}}}}\nu_M(E_{k_1,k_2})\\
        &=1-(1+O(R(N,M)))\sum_{\substack{k_1,k_2\mid \Pam\\ (k_1,k_2)=1 \\
        |k_1k_2|\le q^{\frac{N\theta}{2}}}}\sigma_M(E_{k_1,k_2})\\
        &=\sum_{\substack{k_1,k_2\mid \Pam\\ (k_1,k_2)=1 \\
        |k_1k_2|>q^{\frac{N\theta}{2}}}}\sigma_M(E_{k_1,k_2}) +O(R(N,M))
        =O(R(N,M)).
    \end{align*}
    Now for any $A\in \mathcal{G}$ we have
    $A=\bigsqcup_{E_{k_1,k_2}\in \Ecal_A}E_{k_1,k_2}$ with
    $\Ecal_A\subset \Ecal$. Set
    \[
    \Ecal_A^+=\left\{E_{k_1,k_2}\in\Ecal_A : |k_1k_2|>q^{\frac{N\theta}{2}}\right\}.
    \]
    Thus
    \begin{align*}
        \nu_M(A)&=\sum_{E_{k_1,k_2}\in \Ecal_A}\nu_M(E_{k_1,k_2}) \\
        &=\sum_{E_{k_1,k_2}\in \Ecal_A\setminus \Ecal_A^+}\nu_M(E_{k_1,k_2})
        + \sum_{E_{k_1,k_2}\in \Ecal_A^+}\nu_M(E_{k_1,k_2}) \\
        &=(1+O(R(N,M)))\sum_{E_{k_1,k_2}\in \Ecal_A\setminus \Ecal_A^+}
        \sigma_M(E_{k_1,k_2}) + \sum_{E_{k_1,k_2}\in \Ecal_A^+}\nu_M(E_{k_1,k_2}) \\
        &=(1+O(R(N,M)))\sigma_M(A) \\
        &\qquad +O\!\left(R(N,M)+ \Bigg|\sum_{\substack{k_1,k_2\mid \Pam\\
        (k_1,k_2)=1\\ |k_1k_2|>q^{\frac{N\theta}{2}}}}\sigma_M(E_{k_1,k_2})\Bigg|
        + \Bigg|\sum_{\substack{k_1,k_2\mid \Pam\\ (k_1,k_2)=1\\
        |k_1k_2|>q^{\frac{N\theta}{2}}}}\nu_M(E_{k_1,k_2})\Bigg|\right)\\
        &=(1+O(R(N,M)))\sigma_M(A)+O\left(R(N,M)\right)=\sigma_M(A)+O(R(N,M)).
    \end{align*}

    So all that remains is to show \Cref{sigma_assumption}.

    We apply Rankin's trick. Let $0<\eta\le\frac{1}{2}$ be a parameter to be
    chosen later. By \cref{def:siftable}(ii)(a) we have $|p|\rho(p)\le1$ for
    every prime $p$, hence $|m|\rho(m)\le1$ for every squarefree $m\in\Mon$;
    so for all $k_1,k_2\mid \Pam$ with
    $(k_1,k_2)=1$ and $|k_1k_2|>q^{\frac{N\theta}{2}}$,
    \[
    1\le \left(|k_1k_2|\rho(k_1k_2)\right)^{-\eta}
    \le q^{-\frac{N\theta\eta}{2}}\rho(k_1k_2)^{-\eta}.
    \]
    
    Hence
    \begin{align*}
     \sum_{\substack{k_1,k_2\mid \Pam\\ (k_1,k_2)=1 \\
     |k_1k_2|>q^{\frac{N\theta}{2}}}}\sigma_M(E_{k_1,k_2})
    &= 
    \sum_{\substack{k_1,k_2\mid \Pam\\ (k_1,k_2)=1 \\
    |k_1k_2|>q^{\frac{N\theta}{2}}}}\rho(k_1k_2)
    \prod_{p\mid \Pam/ k_1k_2}\left(1-2\rho(p)\right) \\
    &\le  q^{-\frac{N\theta\eta}{2}}
    \sum_{\substack{k_1,k_2\mid \Pam\\ (k_1,k_2)=1 }}\rho(k_1k_2)^{1-\eta}
    \prod_{p\mid \Pam / k_1k_2}\left(1-2\rho(p)\right) \\
    &=q^{-\frac{N\theta\eta}{2}}
    \prod_{p\mid \Pam}\left[1+2\left(\rho(p)^{1-\eta}-\rho(p)\right)\right] \\
    &\qquad\qquad
    [\text{Using the same method as in \cref{sigma_is_prob_measure}}] \\
    &\le q^{-\frac{N\theta\eta}{2}}\exp{\left[\sum_{p\mid \Pam}2\left(\rho(p)^{1-\eta}-\rho(p)\right)\right]}\\
    &\qquad\qquad
    [\text{Using $1+t\le e^t$ for non-negative $t$}]
    \end{align*}
Now $t^{1-\eta}-t$ is increasing in $0\le t\le (1-\eta)^\frac{1}{\eta}=:g_\eta$. For $\eta$ small $\frac{1}{3} \le g_\eta \le \frac{1}{e}$ and since $q\ne 2$ $\rho(p) \le \frac{1}{3}$ Thus,
\[
\rho(p)^{1-\eta}-\rho(p) \le \frac{1}{|p|}(|p|^\eta-1)\le \eta q^{M\eta}\frac{\log|p|}{|p|}
\]
Thus by choosing $\eta=\frac{1}{M}$ and using Merten's first theorem \ref{Mertens_1st_thm}, we have,
\begin{align*}
    \sum_{\substack{k_1,k_2\mid \Pam\\ (k_1,k_2)=1 \\
     |k_1k_2|>q^{\frac{N\theta}{2}}}}\sigma_M(E_{k_1,k_2})&\le q^{-\frac{N\theta\eta}{2}}\exp{\left(2\eta q^{M\eta}\sum_{p\mid \Pam}\frac{\log|p|}{|p|}\right)}\\
     &\ll \exp{\left(-\frac{N\theta\eta}{2}\log q  +2\eta q^{M\eta}M\log q\right)} \\
     &=\exp{\left(-\frac{\theta N \log q}{2M}+2q\log q\right)}\\
     &\ll \exp{\left(-\Delta\log M\right)=M^{-\Delta}\le R(N,M)}
\end{align*}
\end{proof}

\section{The Kubilius model and its Poisson approximation}
\label{poisson}

We finally have everything we need to write the law for
$(\omega_{M},\omega_{M}\circ T_a): \Fset(N;a) \rightarrow \mathbb{R}^2$ explicitly and to write its Kubilius
approximation. For each pair $(l_1,l_2) \in \mathbb{Z}^2$ we define the set
$S(l_1,l_2)\in \mathcal{G}$ by
\begin{align*}
S(l_1,l_2)&=\bigcup_{\substack{k_1,k_2\mid\Pam,\ (k_1,k_2)=1 \\
\omega(k_1)=l_1,\ \omega(k_2)=l_2 }}E_{k_1,k_2} \\
& \!\!\!\!\!\!\!\!\!\!\!\!=\left\{f\in \Mon \mid \deg f\le N,\ f\in S,\ f+a\in S,\ (f,a)=1,
\ \omega_{M}(f)=l_1,\ \omega_{M}(f+a)=l_2\right\}
\end{align*}
Now we define the atomic measures $\nu_M^*$ and $\sigma_M^*$ on $\mathbb{R}^2$ by
\begin{align*}
    \nu_M^*\left(\{(l_1,l_2)\}\right) & := \nu_M(S(l_1,l_2)), \\
    \sigma_M^*\left(\{(l_1,l_2)\}\right) & := \sigma_M(S(l_1,l_2)),
\end{align*}
extended additively, and set to $0$ outside the lattice. It is immediate that
\[
\left| \nu_M^*\left(\{(l_1,l_2)\}\right)
-  \sigma_M^*\left(\{(l_1,l_2)\}\right)\right| \ll R(N,M).
\]
Now it is easy to see that $\nu_M^*$ is the law of the random vector
$(\omega_{M},\omega_{M}\circ T_a)$. Indeed, if $\mathbb{P}_{N,S,a}$ denotes the
`normalised' counting measure on $\Fset(N;a)$, then
\[
\nu_M^*(\{(l_1,l_2)\})
=\mathbb{P}_{N,S,a}\circ (\omega_{M},\omega_{M}\circ T_a)^{-1}
\left(\{(l_1,l_2)\}\right).
\]

Now we construct a random variable for which $\sigma^*_M$ is the law. For each
$p\mid \Pam$, let $X_p:=(B_p,B_p')$ denote independent bivariate Bernoulli
random vectors such that
\[
X_p =
\begin{cases}
    (1,0) & \text{with probability } \rho(p), \\
    (0,1) & \text{with probability } \rho(p), \\
    (0,0) & \text{with probability } 1-2\rho(p),
\end{cases}
\]

and set
\[
\Sigma_M:=\sum_{\deg p\le M}X_p .
\]
It is easy to see that $\Sigma_M$ is the random vector for which $\sigma^*_M$
is the law.

Now we use the well known method of approximating this sum
of bivariate Bernoulli vectors by a bivariate Poisson vector.

\begin{definition}
    Given two measures $\mu_1,\mu_2$ on a $\sigma$-algebra $\mathcal{B}$, the \emph{total variation distance} between $\mu_1$ and $\mu_2$ is defined as
    \[
    \dTV(\mu_1,\mu_2):= \sup_{A\in \mathcal{B}} \left| \mu_1(A)-\mu_2(A)\right| .
    \]
\end{definition}

\begin{lemma}[\cite{Roos1999MetricMP}]
\label{roos}
    Let $\{X_k\}_{k \geq 1}$ be a sequence of independent Bernoulli random $2$-vectors, and
    let $\mathbb{P}(X_k = (1,0)) = p_{k,1}$, $\mathbb{P}(X_k = (0,1)) = p_{k,2}$.
    Set $p_k := p_{k,1} + p_{k,2}$ and suppose that
    $\mathbb{P}(X_k = (0,0)) = 1 - p_k$. Put $S_n := \sum_{1 \leq k \leq n} X_k$
    and $\lambda_n(j) := \sum_{1 \leq k \leq n} p_{k,j}$ for $j = 1,2$.

    Let $Z = (Z_1,Z_2)$ be a random $2$-vector where $Z_1$ and $Z_2$ are
    independent Poisson random variables with respective parameters
    $\lambda_n(1)$ and $\lambda_n(2)$. Then
\[
\dTV\!\bigl(\mathcal{L}(S_n), \mathcal{L}(Z)\bigr)
= \sup_{A \subseteq \mathbb{R}^2}
\left| \mathcal{L}(S_n)(A) - \mathcal{L}(Z)(A) \right|
\ll \sum_{j=1}^{2} \min\!\left\{1,\lambda_n(j)^{-1}\right\}
\sum_{1 \leq k \leq n} p_{k,j}^{\,2}.
\]
\end{lemma}

Thus if we take $n=\#\{$primes of degree $\le M\}$ and
$p_{p,1}=p_{p,2}=\rho(p)$ for each prime $p$ of degree $\le M$
, then
$\lambda_n(1)=\lambda_n(2)=\sum_{\deg p\le M}\rho(p)=\lambda_{\rho}(M)$,
$S_n=\Sigma_M$ and $\sigma^*_M=\mathcal{L}(S_n)$. So \cref{roos} gives
\begin{align}
\label{sigma_and_poisson_ineq}
    \dTV(\sigma^*_M,\mathcal{L}(Z))
    & \ll 2\min\{1,\lambda_\rho(M)^{-1}\}\sum_{|p|\le q^M}\rho(p)^2
    \notag \\
    & \le 2\min\{1,\lambda_\rho(M)^{-1}\}\sum_{|p|\le q^M}\frac{1}{|p|^2}
    \ \ll \ \min\{1,\lambda_\rho(M)^{-1}\},
\end{align}

where $Z=(Z_1,Z_2)$ with $Z_1,Z_2$ i.i.d.\ $\Poisson(\lambda_\rho(M))$.

Now we state the most important theorem of this section.

\begin{theorem}
\label{nu_and_Poisson_ineq}
    Let $Z=(Z_1,Z_2)$ be a random vector taking values in $\mathbb{Z}^2$, where
    $Z_1,Z_2$ are i.i.d.\ $\Poisson(\lambda_\rho(M))$. Then
    \[
    \dTV(\nu_M^*,\mathcal{L}(Z))\ll
    \min \left\{1,\frac{1}{\lambda_\rho(M)}\right\}
    +M^{-\Delta}+\frac{1}{N}.
    \]
\end{theorem}

\begin{proof}
    We know that $\dTV$ is a metric on the set of all probability measures.
    Thus by the triangle inequality,
    \[
    \dTV(\nu_M^*,\mathcal{L}(Z))\le   \dTV(\nu_M^*,\sigma_M^*)
    +\dTV(\sigma_M^*,\mathcal{L}(Z)).
    \]
    Since $\sigma_M$ and $\nu_M$ are supported on the sets of $\mathcal{G}$,
    \cref{sigma_nu_distance}
     implies
    \begin{align*}
    \dTV(\nu_M^*,\sigma_M^*)
    &= \sup_{A \in \Bor(\mathbb{R}^2)} \left| \nu_M^*(A) - \sigma_M^*(A) \right|
    \leq \sup_{B \in \mathcal{G}} \left| \nu_M(B) - \sigma_M(B) \right|\\
    &\ll R(N,M)=M^{-\Delta}+\frac{1}{N}.
    \end{align*}
    
    This together with \Cref{sigma_and_poisson_ineq} gives the desired result.
\end{proof}

\section{The bivariate Erd\H{o}s--Kac theorem and truncated Chowla for
polynomials}
\label{sec:Aprox}

Since our Erd\H{o}s--Kac theorem is in terms of the characteristic function, we
first observe the following.

\begin{proposition}
\label{char_fn}
    Let $\varphi^*_{N,S,a}(u,v;M)$ be the characteristic function of the random
    vector $(\widetilde{\omega}_{M},\widetilde{\omega}_{M}\circ T_a)$ ,where $\widetilde{\omega}_M(f)=\frac{\omega_M(f)-\lambda_\rho(M)}{\sqrt{\lambda_\rho(M)}}$. Then
    \[
    \varphi^*_{N,S,a}(u,v;M)=\frac{1}{|\Fset(N;a)|}
    \sum_{\substack{|f|\le q^N \\ (f,a)=1}}\ind_S(f)\ind_S(f+a)
    \exp\!\left\{iu\widetilde{\omega}_{M}(f)
    +iv\widetilde{\omega}_{M}(f+a)\right\}.
    \]
\end{proposition}

\begin{proof}
    For $l\in \mathbb{Z}$, let
    $\widetilde{l}:=\frac{l-\lambda_\rho(M)}{\sqrt{\lambda_\rho(M)}}$ and
    $\widetilde{\mathbb{Z}} =\{\widetilde{l} \mid l \in \mathbb{Z} \}$.
    Furthermore, let $\widetilde{\nu}^*_M:\mathbb{R}^2\longrightarrow [0,1]$ be
    the atomic measure such that
    \[
    \widetilde{\nu}^*_M\left(\{(\widetilde{l}_1,\widetilde{l}_2)\}\right)
    =\nu^*_M\left(\{(l_1,l_2)\}\right).
    \]
    Let $X=(X_1,X_2):=(\widetilde{\omega}_{M},\widetilde{\omega}_{M}\circ T_a)$.
    Then clearly $X$ takes values in $\widetilde{\mathbb{Z}}^2$. If $\zeta$
    denotes the probability mass function of $X$, then
    \[
    \zeta\left(\widetilde{l}_1,\widetilde{l}_2\right)
    = \widetilde{\nu}^*_M\left(\{(\widetilde{l}_1,\widetilde{l}_2)\}\right)
    =\nu^*_M\left(\{(l_1,l_2)\}\right).
    \]
    
    Thus
   \begin{align*}
       \varphi^*_{N,S,a}(u,v;M)
       &=\mathbb{E}\left[\exp\left(i(uX_1+vX_2)\right)\right] = \sum_{(\widetilde{l}_1,\widetilde{l}_2) \in \widetilde{\mathbb{Z}}^2}
       \exp\left(i(u\widetilde{l}_1+v\widetilde{l}_2)\right)
       \nu^*_M\left(\{(l_1,l_2)\}\right) \\
       &=\frac{1}{|\Fset(N;a)|}
       \sum_{(\widetilde{l}_1,\widetilde{l}_2) \in \widetilde{\mathbb{Z}}^2}
       \exp\left(i(u\widetilde{l}_1+v\widetilde{l}_2)\right)
       \!\!\!\!\!\!\!\!\sum_{\substack{\deg f\le N \\ \omega_{M}(f)=l_1,\ \omega_{M}(f+a)=l_2 \\
       (f,a)=1}}\!\!\!\!\!\!\!\!\ind_S(f)\ind_S(f+a) \\
       &=\frac{1}{|\Fset(N;a)|}\sum_{\substack{|f|\le q^N \\ (f,a)=1}}
       \ind_S(f)\ind_S(f+a)\exp\!\left\{iu\widetilde{\omega}_{M}(f)
       +iv\widetilde{\omega}_{M}(f+a)\right\}. \qedhere
   \end{align*}
\end{proof}

The next theorem is essentially the bivariate Erd\H{o}s--Kac theorem for
polynomials.

\begin{theorem}
\label{Erdos_Kac}
    Let $u,v\in \mathbb{R}$ with $|u|,|v|\le \pi \sqrt{\lambda_\rho(M)}$. Then
    \[
    |\varphi^*_{N,S,a}(u,v;M)-\chi(u,v)|\ll e^{-\frac{1}{2}(u^2+v^2)}
    \min\left\{\!1,\frac{|u|^3+|v|^3}{\sqrt{\lambda_\rho(M)}}\!\right\}
    + \min\left\{\!1,\frac{1}{\lambda_\rho(M)}\!\right\} +R(N,M),
    \]
    where $\chi(u,v)=e^{-\frac{1}{2}(u^2+v^2)}$ is the characteristic function
    of the bivariate standard normal distribution.
\end{theorem}

To prove this we use the triangle inequality to get
\begin{equation}
\label{erdos_kac_triangle_ineq}
|\varphi^*_{N,S,a}(u,v;M)-\chi(u,v)|
\le |\varphi^*_{N,S,a}(u,v;M)-\varphi_{\widetilde{Z}}(u,v)|
+|\varphi_{\widetilde{Z}}(u,v)-\chi(u,v)|,
\end{equation}
where $\widetilde{Z}=(\widetilde{Z_1},\widetilde{Z_2})$ is the normalised form
of the Poisson vector defined in the previous section, i.e.
\[
\widetilde{Z}=\left(\frac{Z_1-\lambda_\rho(M)}{\sqrt{\lambda_\rho(M)}},
\frac{Z_2-\lambda_\rho(M)}{\sqrt{\lambda_\rho(M)}}\right).
\]

Now the first term in \Cref{erdos_kac_triangle_ineq} will be estimated by the
work in the previous sections (in fact everything up until now was done for
this very estimation) together with the following standard result.

\begin{lemma}
\label{tv_char}

    Let $(\Omega_1,\mathcal{B}_1,P_1)$ and $(\Omega_2,\mathcal{B}_2,P_2)$ be
    probability spaces, and let
    $X : \Omega_1 \to \mathbb{R}^n$ and $Y : \Omega_2 \to \mathbb{R}^n$ be
    random vectors. Let $\mu_X$ and $\mu_Y$ denote the laws of $X$ and $Y$
    respectively, and let $\varphi_X$ and $\varphi_Y$ be their respective
    characteristic functions. Then
\[
\sup_{u \in \mathbb{R}^n} \left| \varphi_X(u)-\varphi_Y(u) \right|
\leq 2\pi\, \dTV(\mu_X,\mu_Y).
\]
\end{lemma}

\begin{proof}
    See \cite{MANGEREL_2020}, Lemma~2.13.
\end{proof}

So we have, for $|u|,|v|\le \pi \sqrt{\lambda_\rho(M)}$,
\begin{align*}
    |\varphi^*_{N,S,a}(u,v;M)-\varphi_{\widetilde{Z}}(u,v)|
    &\le \sup_{(u,v) \in \mathbb{R}^2}
    |\varphi^*_{N,S,a}(u,v;M)-\varphi_{\widetilde{Z}}(u,v)|
    \le 2\pi\, \dTV\!\left(\widetilde{\nu}^*_M,\mathcal{L}(\widetilde{Z})\right)\\
    &=2\pi \sup_{(\tilde{l}_1,\tilde{l}_2)\in \widetilde{\mathbb{Z}}^2}
    \left| \widetilde{\nu}^*_M\left(\{(\tilde{l}_1,\tilde{l}_2)\}\right)
    -\mathcal{L}(\widetilde{Z})\left(\{(\tilde{l}_1,\tilde{l}_2)\}\right)\right| \\
    &=2\pi \sup_{(l_1,l_2)\in \mathbb{Z}^2}
    \left| \nu^*_M\left(\{(l_1,l_2)\}\right)
    -\mathcal{L}(Z)\left(\{(l_1,l_2)\}\right)\right| \\
    &=2\pi\, \dTV\!\left(\nu^*_M,\mathcal{L}(Z)\right) \ll R(N,M),
\end{align*}

where we used \cref{nu_and_Poisson_ineq} in the last line.

The estimation of the second term in \Cref{erdos_kac_triangle_ineq} is exactly
the same as in \cite{MANGEREL_2020}. Indeed,
\[
|\varphi_{\widetilde{Z}}(u,v)-\chi(u,v)|\ll  e^{-\frac{1}{2}(u^2+v^2)}
\min\left\{1,\frac{|u|^3+|v|^3}{\sqrt{\lambda_\rho(M)}}\right\}
+ \min\left\{1,\frac{1}{\lambda_\rho(M)}\right\};
\]
see \cite{MANGEREL_2020}, (2.53), for the detailed calculation. Thus
\cref{Erdos_Kac} follows immediately.

Now it is finally time to prove our main result.

\begin{proof}[Proof of \Cref{main_thm}]
   We have
   \[
   \left|\sum_{\deg f\le N} \mu_{M}(f)\mu_{M}(f+a)\right|
   =\left|\sum_{cd=a}\ \sum_{|g|\le \frac{q^N}{|d|}}
   \mu_{M}(g)\mu_{M}(g+c)\ind_{(g,c)=1}\right|
   \]
   Now let $N_d=N-\deg d$. We next rewrite the inner sum as follows:
    \begin{align}
    \label{main_thm_eq_2}
        &\left|\sum_{\deg g\le N_d}
        \mu_{M}(g)\mu_{M}(g+c)\ind_{(g,c)=1}\right|
        = \left| \sum_{\deg g\le N_d}\mu^2(g)\mu^2(g+c)\ind_{(g,c)=1}
        e^{i\pi \omega_{M}(g)}e^{i\pi \omega_{M}(g+c)} \right|  \notag  \\
        &\qquad =\left| \sum_{\deg g\le N_d}\mu^2(g)\mu^2(g+c)\ind_{(g,c)=1}
        e^{i\pi\sqrt{\lambda_\rho(M)} \widetilde{\omega}_{M}(g)}
        e^{i\pi\sqrt{\lambda_\rho(M)} \widetilde{\omega}_{M}(g+c)} \right| \notag \\
       &\qquad =|\Fset(N_d;c)|\left|\varphi^*_{N_d,S,c}
       \left(\pi\sqrt{\lambda_\rho(M)},\pi\sqrt{\lambda_\rho(M)};M\right)\right|
       \notag  \\
       &\qquad \le |\Fset(N_d;c)|\left|\varphi^*_{N_d,S,c}
       \left(\pi\sqrt{\lambda_\rho(M)},\pi\sqrt{\lambda_\rho(M)};M\right)
       -\chi \left(\pi\sqrt{\lambda_\rho(M)},\pi\sqrt{\lambda_\rho(M)}\right)
       \right| \notag \\
       &\qquad \qquad + |\Fset(N_d;c)|\left|\chi
       \left(\pi\sqrt{\lambda_\rho(M)},\pi\sqrt{\lambda_\rho(M)}\right)\right| .
    \end{align}
    
The second term is
\begin{align*}
    &|\Fset(N_d;c)|\left|\chi
    \left(\pi\sqrt{\lambda_\rho(M)},\pi\sqrt{\lambda_\rho(M)}\right)\right|
    \asymp q^{N_d} e^{-\pi^2\lambda_\rho(M)}
    \ \ll \ \frac{q^N}{|d|}M^{-\pi^2 },
\end{align*}

while the first term can be computed using \cref{Erdos_Kac} and the calculation
above:
\begin{align*}
    &|\Fset(N_d;c)|\left|\varphi^*_{N_d,S,c}
    \left(\pi\sqrt{\lambda_\rho(M)},\pi\sqrt{\lambda_\rho(M)};M\right)
    -\chi \left(\pi\sqrt{\lambda_\rho(M)},\pi\sqrt{\lambda_\rho(M)}\right)\right| \\
    &\qquad \ll \frac{q^N}{|d|}\left(M^{-\pi^2} + \frac{1}{\log M}
    + M^{-\Delta}+\frac{1}{N}\right).
\end{align*}
Here we have applied \cref{Erdos_Kac} with $N_d=N-\deg d$ in place of $N$;
this is legitimate because $\deg d\le\deg a=O_a(1)$, so that $N_d$ still
satisfies \Cref{main_thm_range} for $N$ large, and the resulting error
$R(N_d,M) \approx R(N,M)$.
Putting this back into \Cref{main_thm_eq_2},
\begin{align*}
    \left|\sum_{\deg g\le N_d} \mu_{M}(g)\mu_{M}(g+c)\ind_{(g,c)=1}\right|
    \ll \frac{q^N}{|d|}\left(M^{-\pi^2 } + \frac{1}{\log M}
    + M^{-\Delta}+\frac{1}{N}\right),
\end{align*}
and finally, putting this back into the decomposition at the start of this
proof, we get
 \begin{align*}
     \sum_{\deg f\le N} \mu_{M}(f)\mu_{M}(f+a)
     & \ll \sum_{cd=a}\frac{q^N}{|d|}\left(M^{-\pi^2}
     + \frac{1}{\log M} + M^{-\Delta}+\frac{1}{N}\right) \\
     & \le q^N \left(M^{-\pi^2 } + \frac{1}{\log M}
     + M^{-\Delta}+\frac{1}{N}\right) \sum_{cd=a}\frac{1}{|d|} \\
     & \ll_a  q^N \left(M^{-\pi^2 } + \frac{1}{\log M}
     + M^{-\Delta}+\frac{1}{N}\right)
     \ \ll_a\ \frac{q^N}{\log M},
 \end{align*}
Thus we are done.
 \end{proof}

\appendix

\section{Brun's sieve for polynomials}
\label{sec:brun}

We begin by proving two classical results in the polynomial case that will come
in handy in a moment.

\begin{lemma}[Generalised fundamental property of the M\"obius function]
\label{fundamental_mobius}
    For any $f\in \Mon$ and any $s\in \mathbb{N}$ we have
    \[
    \sum_{\substack{d\mid f\\ \omega(d)\leq 2s+1}}\mu(d)
    \le \sum_{d\mid f}\mu(d)=\ind_{f=1}
    \le  \sum_{\substack{d\mid f\\ \omega(d)\leq 2s}}\mu(d).
    \]
\end{lemma}

\begin{proof}
    For $f=1$ this is obvious. For $f\ne 1$, let $\omega(f)=n$. Then
    \[
    \sum_{\substack{d\mid f\\ \omega(d)=r}}\mu(d)= (-1)^r\binom{n}{r},
    \]
    and hence
    \begin{align*}
    G^{(k)}(f):&=  \sum_{\substack{d\mid f\\ \omega(d)\leq k-1}}\mu(d)
    =\sum_{r=0}^{k-1}(-1)^r\binom{n}{r}=(-1)^{k-1}\binom{n-1}{k-1}
    \quad \text{for all } k\in \mathbb{N}^{*},
    \end{align*}
    where the last equality can be proved by induction. The lemma follows
    immediately.
\end{proof}

\begin{lemma}[Mertens' first theorem]
\label{Mertens_1st_thm}
        For any $M\in\mathbb{N}^{*}$,
        \[
        \sum_{|p|\le q^{M}}\frac{\log |p|}{|p|}=M\log q+O(1).
        \]
\end{lemma}

\begin{proof}
         This follows immediately from the prime number theorem for
         polynomials. Indeed, if $a_n$ denotes the number of primes of
         degree~$n$, then by the prime number theorem for polynomials
         $a_n=\frac{q^n}{n}+O\!\left(\frac{q^{n/2}}{n}\right)$. Thus
         \begin{align*}
              \sum_{|p|\le q^{M}}\frac{\log |p|}{|p|}
              &=\sum_{n=1}^M\frac{a_n\log q^n}{q^n}
              =\sum_{n=1}^M\left(\frac{q^n}{n}
              +O\!\left(\frac{q^{n/2}}{n}\right)\right)\frac{n\log q}{q^n}\\
              &=\log q\sum_{n=1}^M\left(1+O(q^{-n/2})\right)
              =M\log q +O(1). \qedhere
         \end{align*}
\end{proof}

\begin{lemma}[Mertens' second theorem]
\label{Mertens_thm}
For any
$M\in \mathbb{N}^{*}$,
   \[
   \sum_{|p|\le q^M}\frac{1}{|p|}=\log M +O(1) + O\!\left(\frac{1}{M}\right)
   \asymp \log M .
   \]
\end{lemma}

\begin{proof}
  Once again, since $a_n=\frac{q^n}{n}+O\!\left(\frac{q^{n/2}}{n}\right)$, we
  have
    \begin{align*}
    \sum_{\deg p\le M}\frac{1}{|p|} &=\sum_{n=1}^M \frac{a_n}{q^n}
    = \sum_{n=1}^M \frac{1}{q^n}\left(\frac{q^n}{n}
    +O\!\left(\frac{q^{n/2}}{n}\right)\right)
    = \sum_{n=1}^M \left(\frac{1}{n}
    +O\!\left(\frac{1}{nq^{n/2}}\right)\right) \\
    &= \log M +\gamma_0 + \frac{1}{2M}-\epsilon_M
    +O\!\left(\sum_{n=1}^M \frac{q^{-n/2}}{n}\right).
     \end{align*}
     But $\sum_{n=1}^M \frac{q^{-n/2}}{n} \le \sum_{n=1}^M q^{-n/2}
     =(\sqrt q -1)^{-1}$, thus $\sum_{\deg p\le M}\frac{1}{|p|} \le \log M +3 \asymp \log M$
\end{proof}

Let $w: \Mon\longrightarrow \mathbb{R}$ be a multiplicative function. We say
that $w$ satisfies the conditions $\Omega_0$ and $\Omega$ if there
exist constants $A_0>0$ and $A>0$ respectively such that:
\begin{gather*}
w(p) \leq \frac{A_0}{|p|} \quad \text{for all primes } p, \tag{$\Omega_0$}\\
0\leq w(p) \leq \frac{A}{|p|+A} \quad \text{for all primes } p. \tag{$\Omega$}
\end{gather*}
It is easy to see that $\Omega \implies \Omega_0$.

If $w(p)\ne1$ $\forall p$ prime, then we define
$w^*: \Mon\longrightarrow \mathbb{R}$ by
\[
w^*(d):= \frac{w(d)}{\prod_{p\mid d} \left(1-w(p)\right)}.
\]
Moreover, note that $w$ satisfies $\Omega$ iff $w^*$ satisfies $\Omega_0$.

Our goal for this section is to estimate
\begin{equation}
\label{lambda_sum}
\sum_{d\mid\Pam}\gamma^{(r)}(d)w(d),
\end{equation}
where the sieve weights $\gamma^{(r)}: \Mon \longrightarrow \mathbb{R}$ are
defined similarly to the ones in Brun's pure sieve, as follows:
\[
\gamma ^{(r)}(d)=
\begin{cases}
  \mu(d)   &  \text{if } \omega(d)<r,\\
  0 &  \text{otherwise,}
\end{cases}
\]
where $r$ is a parameter to be chosen later.

Note that $G^{(r)}(f)=\sum_{d\mid f} \gamma ^{(r)}(d)$, so by M\"obius
inversion,
\[
\gamma^{(r)}(d)= \sum_{d_1\mid d}\mu\!\left(\frac{d}{d_1}\right)G^{(r)}(d_1).
\]
Thus
\begin{align*}
\sum_{d\mid\Pam}\gamma^{(r)}(d)w(d)
&=  \sum_{d\mid\Pam}w(d)\sum_{d_1\mid d}
\mu\!\left(\frac{d}{d_1}\right)G^{(r)}(d_1)= \sum_{d_1\mid\Pam}G^{(r)}(d_1)w(d_1)\sum_{t\mid\frac{\Pam}{d_1}}\mu(t)w(t) \\
&= \sum_{d_1\mid\Pam}G^{(r)}(d_1)w(d_1)
\prod_{p\mid\frac{\Pam}{d_1}}\left(1-w(p)\right) \\
&= W(M)\sum_{d_1\mid\Pam}G^{(r)}(d_1)w^*(d_1)
\quad \left[\text{where } W(M):= \prod_{p\mid\Pam}\left(1-w(p)\right)\right] \\
&= W(M) \left[ 1+ \sum_{1\ne d\mid\Pam}G^{(r)}(d)w^*(d) \right].
\end{align*}
We thus have the following.

\begin{proposition}
\label{lambda_1}
    For any multiplicative function $w$ such that $w(p)\ne 1$ for all primes
    $p$, and any $r\in \mathbb{N}^{*}$,
    \[
    \sum_{d\mid\Pam}\gamma^{(r)}(d)w(d)
    = W(M) \left[ 1+ \sum_{1\ne d\mid\Pam}G^{(r)}(d)w^*(d) \right].
    \]
\end{proposition}

Now, to estimate the sum on the right-hand side, recall that
$\left|G^{(r)}(f)\right| = \binom{\omega(f)-1}{r-1}\le \binom{\omega(f)}{r}$.
Hence
\begin{align}
\label{sigma_ineq}
\left|\sum_{1\ne d\mid\Pam}G^{(r)}(d)w^*(d) \right|
& \le \sum_{1\ne d\mid\Pam}\binom{\omega(d)}{r}w^*(d) \notag = \sum_{m=r}^{\omega(\Pam)}\binom{m}{r}
    \sum_{\substack{1\ne d\mid\Pam \\ \omega(d)=m}}w^*(d) \notag \\
    &\le \sum_{m=r}^{\infty} \binom{m}{r}\frac{1}{m!}
    \left(\sum_{p\mid\Pam}w^*(p)\right)^m  \notag \\
    &=\frac{1}{r!}\left(\sum_{p\mid\Pam}w^*(p)\right)^r
    \exp\left(\sum_{p\mid\Pam}w^*(p)\right).
\end{align}

Suppose now that $w^*$ satisfies $\Omega_0$ 
so that $w^*(p)\le\frac{A_0}{|p|}$. Then, by
Mertens' second theorem (\Cref{Mertens_thm}),
\begin{equation}
\label{V_def}
V:=\sum_{p\mid\Pam}w^*(p)\ \le\ A_0\sum_{\deg p\le M}\frac{1}{|p|}
\ \le A_0\left(\log M+3\right),
\end{equation}
and \Cref{sigma_ineq} together with $r!\ge\left(\frac{r}{e}\right)^r$ gives
\begin{equation}
\label{brun_core}
\left|\sum_{1\ne d\mid\Pam}G^{(r)}(d)w^*(d) \right|
\ \le\ \frac{V^r}{r!}e^{V}
\ \le\ \exp\left(-r\log\frac{r}{eV}+V\right).
\end{equation}

Accordingly, fix $\theta>0$ and choose
\begin{equation}
\label{r_choice}
r:=\left\lfloor\frac{\theta N}{4M}\right\rfloor
\quad\text{or}\quad
r:=\left\lfloor\frac{\theta N}{4M}\right\rfloor-1,
\end{equation}
according to whether we want an upper or a lower sieve weight; indeed, it
follows from the definitions of $G^{(r)}$ and $\gamma^{(r)}$ that
$\gamma^{(r)}$ is a lower sieve weight if $r$ is even and an upper sieve weight
if $r$ is odd. Either choice satisfies $r\le\frac{\theta N}{4M}$, so that
$\omega(d)<r$ and $d\mid\Pam$ force
\begin{equation}
\label{brun_support}
\deg d\ <\ Mr\ \le\ \frac{\theta N}{4}:
\end{equation}
the sieve weights are supported on $\deg d\le\frac{\theta N}{4}$.

\begin{proposition}
\label{brun_quantitative}
    Fix any $\theta>0$ and $\Delta>0$. There is a constant $A_0=A_0(w)$ and
    $C_0=C_0(\theta,A_0,\Delta,q)$ with the following property. If
    $\log M\ge max\{3,\frac{40}{\theta C_0}\}$ and
    \begin{equation}
    \label{brun_hypothesis}
    \frac{N}{M}\ \ge\ C_0\log M,
    \end{equation}
    then for every multiplicative $w$ satisfying $\Omega$ and every
    $r$ as in \Cref{r_choice},
    \[
    \left|\sum_{1\ne d\mid\Pam}G^{(r)}(d)w^*(d) \right|\ \le\ M^{-\Delta}.
    \]
\end{proposition}

\begin{proof}
    \Cref{V_def} gives
    $V\le 2A_0\log M$. By \Cref{r_choice} and \Cref{brun_hypothesis},
    \[
    r\ \ge\ \frac{\theta N}{4M}-2\ \ge\ \frac{\theta C_0\log M}{4}-2
    \ \ge\ \frac{\theta C_0\log M}{5}
    \]
    since $\log M\ge\frac{40}{\theta C_0}$. Hence
    \[
    \frac{r}{eV}\ \ge\ \frac{\theta C_0\log M/5}{2eA_0\log M}
    \ =\ \frac{\theta C_0}{10eA_0}\ =:\ \Lambda ,
    \]
    a quantity independent of $M$ and $N$. Substituting into
    \Cref{brun_core},
    \[
    \left|\sum_{1\ne d\mid\Pam}G^{(r)}(d)w^*(d) \right|
    \le \exp\left(-\frac{\theta C_0}{5}\log\Lambda\cdot\log M
    +2A_0\log M\right)
    = M^{-\Delta(C_0)},
    \]
    where $\Delta(C_0):=\frac{\theta C_0}{5}\log\frac{\theta C_0}{10eA_0}
    -2A_0$. Since $\Delta(C_0)\to\infty$ as $C_0\to\infty$, we may choose
    $C_0$ so that $\Delta(C_0)\ge\Delta$.
\end{proof}

Combining \cref{lambda_1}, \Cref{brun_support} and
\cref{brun_quantitative} we obtain the main result of this section.

\begin{theorem}
\label{sieve_main_thm}
    Fix any $\theta >0$ and $\Delta>0$, and let $C_0$ be as in
    \cref{brun_quantitative}. Let $M,N\in\mathbb{N}^{*}$ satisfy $\log M\ge max\{3,\frac{40}{\theta C_0}\}$
    and $\frac{N}{M}\ge C_0\log M$. Then there exist sieve weights
    $\gamma^+,\gamma^-$, supported on $\{d\mid\Pam:\deg d\le\frac{\theta
    N}{4}\}$ and satisfying
    $(1*\gamma^-)(d)\le\ind_{d=1}\le(1*\gamma^+)(d)$, such that for every
    multiplicative $w:\Mon\longrightarrow\mathbb{R}$ obeying
    $\Omega$,
    \begin{equation}
        \sum_{\substack{d\mid\Pam \\ \deg d\le\frac{\theta N}{4}}}
        \gamma^{s}(d)w(d)
        = W(M) \left[ 1+ O\!\left(M^{-\Delta}\right) \right],
        \qquad  s\in \{+,-\},
    \end{equation}
    uniformly in $a$ and in $w$.
\end{theorem}

\begin{remark}
\label{rem:brun_barrier}
The hypothesis $\frac{N}{M}\ge C_0\log M$  is implied by $M\le\frac{N}{C_0\log N}$ --- which is implied by
the upper bound appearing in \Cref{main_thm_range}. This is the natural limit
of the pure sieve. Indeed \Cref{brun_core} yields a saving only when
$r>eV$, and $r$ is capped at $\frac{\theta N}{4M}$ by the support
constraint \Cref{brun_support}, while $V\asymp A_0\log M$ by \Cref{V_def};
so $\theta N>4eA_0M\log M$ is forced. At $M=\frac{N}{\log N}$ one has
$\frac{N}{M}=\log N\asymp\log M$ and the sieve gives nothing at all.

If we compromise on the range of $M$ allowing $M\ll N^\epsilon$, it follows from \Cref{brun_core} that the error can be taken to be $e^{-\theta\beta\log \beta}$ with $\beta:=\frac{N}{M}$. But reaching $M=o(N)$ would require replacing the pure sieve by a fundamental
lemma, whose error $e^{-\theta\beta\log \beta}$ has an implied constant
depending on the sieve dimension rather than on $\sum_p\rho(p)$. The dimension
here is~$1$: by \cref{Mertens_1st_thm},
\[
\sum_{\deg p\le M}\rho(p)\log|p|=\sum_{\deg p\le M}\frac{\log|p|}{|p|}+O(1)
=M\log q+O(1)=\log z+O(1),
\]
where $z=q^M$. A dimension-one fundamental lemma would therefore need only
$\beta\to\infty$, giving $M=o(N)$. We do not pursue this here.
\end{remark}

\section{Independence of $\mu_M$ and $\mu$}
\label{sec:independence}

Philosophically it is clear that the
parities of $\omega_M(f)$ and
$\omega(f)$ are uncorrelated. This can be put in a quantitative way as well. We
wish to show that
\[
\sum_{\deg f \le N}\mu(f)\mu_M(f)=o(q^N).
\]
Define $\xi_M(f)=\mu(f)\mu_M(f)$. This is clearly multiplicative, and
\[
\xi_M(p)=
\begin{cases}
    1 &\text{if } \deg p\le M,\\
    -1 &\text{if } \deg p>M.
\end{cases}
\]

Let
$D_M(s):=\sum_{f\in \Mon}\frac{\xi_M(f)}{|f|^s}$ for $\mathfrak{Re}(s)>1$
be the Dirichlet series associated to it. Then for $\mathfrak{Re}(s)>1$ we get
the Euler product
\begin{align*}
    D_M(s) &=\prod_{\deg p \le M }\left(1+\frac{1}{|p|^s}\right)
    \prod_{\deg p'>M}\left(1-\frac{1}{|p'|^s}\right) \\
    &=\prod_{\deg p\le M}\left(1+\frac{1}{|p|^s}\right)
    \left(1-\frac{1}{|p|^s}\right)^{-1}
    \prod_{p'}\left(1-\frac{1}{|p'|^s}\right)
    =\frac{A_M(s)}{\zeta_{\A}(s)},
\end{align*}
where $A_M(s):=\prod_{\deg p\le M}\left(1+\frac{1}{|p|^s}\right)
\left(1-\frac{1}{|p|^s}\right)^{-1}$ and $\zeta_{\A}(s)$ is the zeta function
of $\A=\mathbb{F}_q[t]$. Clearly $D_M$ can be analytically continued, with a zero
at $s=1$. Thus, putting $u=q^{-s}$,
\[
\overline{D}_M(u):=D_M(s)=\sum_{n\ge 0}
\left(\sum_{\deg f=n}\xi_M(f)\right)u^n
\]
converges for $u=1/q$. So
\begin{align*}
    \sum_{\deg f=n}\xi_M(f)&=o(q^n) \implies\ \sum_{\deg f\le N}\xi_M(f)=o(q^N)\implies\ \sum_{\deg f \le N}\mu(f)\mu_M(f)=o(q^N),
\end{align*}

as we wanted.
\section*{Acknowledgments}
I thank my Ph.D. advisor, Prof. Anirban Mukhopadhyay, for the helpful discussions and his valuable insights throughout this work.
\bibliographystyle{abbrvnat}
\bibliography{refs}

\begin{thebibliography}{11}
\providecommand{\natexlab}[1]{#1}
\providecommand{\url}[1]{\texttt{#1}}
\expandafter\ifx\csname urlstyle\endcsname\relax
  \providecommand{\doi}[1]{doi: #1}\else
  \providecommand{\doi}{doi: \begingroup \urlstyle{rm}\Url}\fi

\bibitem[Carmon(2015)]{Carmon_2015}
D.~Carmon.
\newblock The autocorrelation of the {M}\"obius function and {C}howla's conjecture for the rational function field in characteristic 2.
\newblock \emph{Philos. Trans. Roy. Soc. A}, 373\penalty0 (2040):\penalty0 20140311, 14, 2015.
\newblock URL \url{https://doi.org/10.1098/rsta.2014.0311}.

\bibitem[Carmon and Rudnick(2014)]{Carmon_Rudnick}
D.~Carmon and Z.~Rudnick.
\newblock The autocorrelation of the {M}\"obius function and {C}howla's conjecture for the rational function field.
\newblock \emph{Q. J. Math.}, 65\penalty0 (1):\penalty0 53--61, 2014.
\newblock URL \url{https://doi.org/10.1093/qmath/has047}.

\bibitem[Chowla(1965)]{ChowlaConj1965}
S.~Chowla.
\newblock \emph{The {Riemann} Hypothesis and {Hilbert's} Tenth Problem}.
\newblock Gordon and Breach, New York, 1965.

\bibitem[Mangerel(2020)]{MANGEREL_2020}
A.~P. Mangerel.
\newblock On the bivariate {E}rd\H{o}s-{K}ac theorem and correlations of the {M}\"obius function.
\newblock \emph{Math. Proc. Cambridge Philos. Soc.}, 169\penalty0 (3):\penalty0 547--605, 2020.
\newblock URL \url{https://doi.org/10.1017/s0305004119000288}.

\bibitem[Matom{\"a}ki and Radziwi{\l}{\l}(2017)]{matomaki2017multiplicativefunctionsshortintervals}
K.~Matom{\"a}ki and M.~Radziwi{\l}{\l}.
\newblock Multiplicative functions in short intervals, 2017.
\newblock URL \url{https://arxiv.org/abs/1501.04585}.

\bibitem[Matomäki et~al.(2015)Matomäki, Radziwiłł, and Tao]{Matom_ki_2015}
K.~Matomäki, M.~Radziwiłł, and T.~Tao.
\newblock An averaged form of chowla’s conjecture.
\newblock \emph{Algebra \&; Number Theory}, 9\penalty0 (9):\penalty0 2167–2196, Nov. 2015.
\newblock ISSN 1937-0652.
\newblock URL \url{http://dx.doi.org/10.2140/ant.2015.9.2167}.

\bibitem[Ramar{\'e}(2018)]{Ramare2018}
O.~Ramar{\'e}.
\newblock Chowla's conjecture: from the {Liouville} function to the {M\"obius} function.
\newblock In \emph{Ergodic Theory and Dynamical Systems in their Interactions with Arithmetics and Combinatorics (S. Ferenczi, J. Ku{\l}aga-Przymus and M. Lema{\'n}czyk, eds.), CIRM Jean-Morlet Chair, Fall 2016}, pages 317--323. Springer International Publishing, Cham, 2018.
\newblock ISBN 978-3-319-74908-2.
\newblock \doi{10.1007/978-3-319-74908-2_16}.

\bibitem[Roos(1998)]{Roos1999MetricMP}
B.~Roos.
\newblock Metric multivariate {P}oisson approximation of the generalized multinomial distribution.
\newblock \emph{Teor. Veroyatnost. i Primenen.}, 43\penalty0 (2):\penalty0 404--413, 1998.
\newblock URL \url{https://doi.org/10.1137/S0040585X97976921}.

\bibitem[Sawin and Shusterman(2022)]{sawin2019chowlatwinprimesconjectures}
W.~Sawin and M.~Shusterman.
\newblock On the {C}howla and twin primes conjectures over {$\Bbb{F}_q[T]$}, 2022.
\newblock URL \url{https://doi.org/10.4007/annals.2022.196.2.1}.

\bibitem[Tao(2016)]{TAO_2016}
T.~Tao.
\newblock The logarithmically averaged {C}howla and {E}lliott conjectures for two-point correlations.
\newblock \emph{Forum Math. Pi}, 4:\penalty0 e8, 36, 2016.
\newblock URL \url{https://doi.org/10.1017/fmp.2016.6}.

\bibitem[Tao and Ter\"av\"ainen(2018)]{Tao_2019}
T.~Tao and J.~Ter\"av\"ainen.
\newblock Odd order cases of the logarithmically averaged {C}howla conjecture.
\newblock \emph{J. Th\'eor. Nombres Bordeaux}, 30\penalty0 (3):\penalty0 997--1015, 2018.
\newblock URL \url{https://doi.org/10.5802/jtnb.1062}.

\end{thebibliography}

\end{document}